\documentclass[11pt]{amsart}
\usepackage{amsfonts}
\usepackage{amsmath}
\usepackage[mathscr]{eucal}
\usepackage{amscd, latexsym, graphicx, mathrsfs, enumerate}
\usepackage{amssymb}
\usepackage{times}
\usepackage{changepage}
\usepackage{caption}
\usepackage[curve,all]{xy}
\xyoption{curve}
\xyoption{line}
\xyoption{arc}
\xyoption{color}
\usepackage{array}
\usepackage{enumitem}
\usepackage{graphicx}
\usepackage[dvipdfmx]{color}
\usepackage[top=25mm,bottom=25mm,left=25mm,right=25mm]{geometry}

\makeindex

\newtheorem{thm}{{{Theorem}}}[section]
\newtheorem{prop}[thm]{{Proposition}}
\newtheorem{lem}[thm]{{Lemma}}
\newtheorem{cor}[thm]{{Corollary}}

\newtheorem{rmk}[thm]{Remark}
\numberwithin{equation}{section}
\newtheorem{Def}[thm]{Definition}

\def\Z{\mathbb{Z}}
\def\Q{\mathbb{Q}}
\def\R{\mathbb{R}}
\def\C{\mathbb{C}}
\def\A{\mathbb{A}}

\def\GL{{\mathop{\mathrm{GL}}}}

\def\Sp{{\mathop{\mathrm{Sp}}}}
\def\GSp{{\mathop{\mathrm{GSp}}}}
\def\PGSp{{\mathop{\mathrm{PGSp}}}}

\def\diag{{\mathop{\mathrm{diag}}}}

\def\ve{\varepsilon}

\def\min{{\mathrm{min}}}

\def\cL{{\mathcal{L}}}

\def\bs{{\backslash}}
\def\ds{\displaystyle}
\def\ord{{\rm ord}}
\def\lra{{\longrightarrow}}
\def\G{{\Gamma}}

\newcommand{\Gc}{{\mathcal G}}

\newcommand{\bF}{\overline{\mathbb{F}}}

\newcommand{\lla}{\longleftarrow}

\newcommand{\vp}{\varphi}

\newcommand{\mL}{\mathcal{L}}
\newcommand{\mM}{\mathcal{M}}

\newcommand{\g}{\gamma}

\newcommand{\calO}{\mathcal{O}}

\newcommand{\calX}{\mathcal{X}}

\newcommand{\bbF}{\mathbb{F}}

\def\MB{{\rm{MB}}}

\def\deg{{\text{deg}}}

\numberwithin{equation}{section}

\begin{document}
\title[Isogeny graphs associated to Moret-Bailly families of supersingular abelian surfaces]
{Isogeny graphs associated to Moret-Bailly families of supersingular abelian surfaces}

%\date{\today}
\keywords{Moret-Bailly family, supersingular abelian surface, algebraic modular form, 
Sarnak-Xue type theorem.
}
\thanks{}
\subjclass[2010]{}

\author{Toshiyuki Katsura}
\address{Toshiyuki Katsura \\
Graduate School of Mathematical Sciences, The University of Tokyo \\
3-8-1, Komaba, Meguro-ku, Tokyo 153-8914, JAPAN}
\email{tkatsura@g.ecc.u-tokyo.ac.jp}

\author{Takuya Yamauchi}
\address{Takuya Yamauchi \\
Mathematical Institute, Tohoku University\\
 6-3, Aoba, Aramaki, Aoba-Ku, Sendai 980-8578, JAPAN}
\email{takuya.yamauchi.c3@tohoku.ac.jp}

\thanks{Research of the first author is partially supported by JSPS Grant-in-Aid 
for Scientific Research (C) No.26K06737.}

\maketitle

\begin{abstract}
For an odd prime $p$ and any prime $\ell\neq p$, we study finite directed graphs arising from the set of all equivalence classes of Moret-Bailly families of abelian surfaces in characteristic $p$ together with relative $(\ell,\ell)$-isogenies. We relate these graphs to the space of algebraic modular forms on an inner form of $GSp_4/\Q$ that is compact modulo its center and, using the Jacquet-Langlands correspondence, estimate the eigenvalues of their adjacency matrices. We further investigate a Sarnak-Xue type theorem in this setting, providing a first step toward the study of cut-off phenomena for these graphs.
\end{abstract}

\tableofcontents

\section{Introduction}\label{intro}

Isogeny graphs are finite graphs associated with 
supersingular elliptic curves, more generally, principally polarized superspecial abelian varieties over finite fields (cf. \cite{Pizer-Ram}, \cite{KT}, \cite{FS}, \cite{ATY},\cite{JZ} among others). They have attracted attention not only in arithmetic geometry \cite{JZ} ,  
but also in cryptography since the objects consist a building block in a prospective secure encryption scheme \cite{CGL}.

In this paper, we study the isogeny graphs associated to Moret-Baily families whose fibers 
are consisting of principally polarized supersingular abelian surfaces.

In previous works, the vertices typically consist of supersingular elliptic curves or
principally polarized superspecial abelian varieties, whereas in our setting they consist of
families of such objects endowed with additional structure.
A novelty of our framework is that it allows us to apply the theory of Siegel (para-)modular forms of degree $2$ via the Jacquet-Langlands correspondence unconditionally, without imposing any condition at the ramified primes.
As a consequence, we obtain effective estimates for the eigenvalues of the random walk matrices
associated with our graphs.
This stands in contrast to earlier cases, where no direct connection to automorphic forms 
(Jacquet-Langlands type correspondence) was known, except in certain special situations, such as isogeny graphs of supersingular elliptic curves in which case Pizer made a good use of 
elliptic modular forms \cite{Pizer-Ram} via the Jacquet-Langlands correspondence.
For this reason, the authors of \cite{ATY} developed a more elaborate approach, constructing a bridge to the Bruhat--Tits building in order to apply Kazhdan's Property (T).

To go into further explanation we need to fix some notation and 
the details are left to the relevant sections.
Let $p$ be an odd prime and $\ell$ be a prime different from $p$. 
Fix an algebraically closed field $\bF_p$ of the finite field $\bbF_p=\Z/p\Z$. 
Let $\mathcal{S}_{2,\bF_p}$ be the set of all isomorphism classes of 
Moret-Baily families so that each fiber of a Moret-Baily family $q:\mathcal{X}\lra \bf P^1$ 
is a principally polarized supersingular abelian surface over $\bF_p$.  
The cardinality of $\mathcal{S}_{2,\bF_p}$  is known (see (\ref{classnumber})) and its main terms is $\ds\frac{p^2-1}{2880}$ for all $p$ large enough. 

The  $(\ell,\ell)$-isogeny graph $\Gc^{{\rm MB}}(\ell,p)$ for 
$\mathcal{S}_{2,\bF_p}$ is defined as a directed graph where
\begin{itemize}
\item the set of vertices $V:=V(\Gc^{{\rm MB}}(\ell,p))$ is $\mathcal{S}_{2,\bF_p}$ and
\item the edges are $(\ell,\ell)$-isogenies (Section \ref{reliso}). Thus, two vertices $v_1,v_2$ are tied by 
an $(\ell,\ell)$-isogeny $f:v_1\lra v_2$. Two edges $f:v_1\lra v_2$ and  
$f':v_1\lra v_2$ are identifies if their kernels are identified by an element of the reduced automorphism group ${\rm RA}(v_1)$ of $v_1$. 
\end{itemize}
Our graph is a finite directed graph but not necessarily regular because of possible non-triviality of 
the reduced automorphism groups of vertices. However, for each edge $f:v_1\lra v_2$, one can assign 
the weight $w(f):=|{\rm RA}(v_1)/{\rm Stab}_{{\rm RA}(v_1)}({\rm Ker}(f))|$ 
(see (\ref{weight}) for precise definition). 
Then, the weighted sum $\ds\sum_{v'\in V,\ f:v\lra v'}w(f)$ for each vertex $v$ satisfies 
$\ds\sum_{v'\in V,\ f:v\lra v'}w(f)=N_2(\ell):=(\ell^2+1)(\ell+1)$. Here we regard $w(f)=0$ if 
$v,v'$ are not tied.

Let $\ell^2(\Gc^{{\rm MB}}(\ell,p))$ be the finite dimensional Hilbert space 
associated to $\Gc^{{\rm MB}}(\ell,p)$ consisting of all complex valued functions of 
$V:=V(\Gc^{{\rm MB}}(\ell,p))$. 
It endows with the inner product defined by 
$$
\langle F_1,F_2 \rangle=\sum_{v\in V}F_1(v)\overline{F_2(v)}\frac{1}{|{\rm RA}(v)|}.
$$
The associated random walk operator $P:\ell^2(\Gc^{{\rm MB}}(\ell,p))
\lra \ell^2(\Gc^{{\rm MB}}(\ell,p)),\ F\mapsto PF$ for $\Gc^{{\rm MB}}(\ell,p)$ is given by 
$$(PF)(v):=\frac{1}{N_2(\ell)}\sum_{v'\in V,\ f:v'\lra v}F(v')w(f)$$
 which corresponds to a normalized Hecke operator 
 $\frac{1}{N_2(\ell)}T(\ell)$ on the space of algebraic modular forms (see Proposition \ref{comp2}). 
 Therefore, $P$ is self-dual with respect to $\langle \cdot,\cdot \rangle$. This fact is not obvious by definition of $P$, but
 since $w(f)$ satisfies the formula analogue to \cite[(3.1)]{FS}, the operator $P$ is reversible in the sense of 
 \cite[Section 1.6]{LP}. Either way, all eigenvalues of the graph Laplacian $\Delta:={\rm id}-P$
are real and belong to the interval $[0,2]$. These facts for the reversible random matrices follow from \cite[Proposition 4.2, Lemma 4.7, see also the proof of Theorem 4.9]{FS}. 

Our first main result is the following:
\begin{thm}\label{main1}{\rm(}Theorem \ref{refinement} with Proposition \ref{comp2}{\rm)} Keep the notation as above. It holds that 
\begin{enumerate}
\item[{\rm(1)}] The graph $\Gc^{{\rm MB}}(\ell,p)$ is connected and not bipartite, 
\item[{\rm(2)}] Any non-zero eigenvalue $\alpha$ of $\Delta$ satisfies 
$$0.22287\ldots \leq 1-\frac{\ell^2+ \ell +2\ell\sqrt{\ell}}{N_{2}(\ell)}\le 
\alpha \le 
1+\frac{4\ell\sqrt{\ell}}{N_{2}(\ell)} \leq 1.75424\ldots
$$
\end{enumerate}
\end{thm} 
A key step in the proof of the theorem is to establish an isomorphism between
$\ell^2(\Gc^{{\rm MB}}(\ell,p))$ and the space $M(U_{{\rm npr}})$,
which consists of algebraic modular forms on an inner form of $GSp_4$ which is compact modulo its center. 
Via the Jacquet-Langlands correspondence established by van Hoften~\cite{vH},
the problem is reduced to the study of Siegel paramodular forms of degree~$2$ and of weight 3 
with respect to the paramodular group $K(p)$.
After a bit careful analysis, we prove an asymptotic Ramanujan property:
\begin{thm}\label{asymp-ram} {\rm(}A part of Theorem \ref{refinement} with Proposition \ref{comp2}{\rm)} 
Fix a prime $\ell$. For each odd prime $p\neq \ell$, let $\alpha_{{\rm ave}}(\ell,p)$ be the 
average of all eigenvalues of the graph Laplacian $\Delta$ of $\Gc^{{\rm MB}}(\ell,p)$. 
Then it holds that 
$$0.24575\ldots \leq 1-\frac{4\ell\sqrt{\ell}}{N_2(\ell)}\le 
\liminf_{p\to\infty \atop p\neq \ell}\alpha_{{\rm ave}}(\ell,p)\le 
\limsup_{p\to\infty \atop p\neq \ell}\alpha_{{\rm ave}}(\ell,p)
\le 1+\frac{4\ell\sqrt{\ell}}{N_2(\ell)} \leq 1.75424\ldots.$$
\end{thm}

As a byproduct in proving the above theorem, we prove a Sarnak-Xue type density hypothesis (cf. \cite{EGGG}), which heuristically explains
that the sum of multiplicity of non-tempered unramified representations of $\PGSp_4(\Q_\ell)$ occurring in
\[
L^2\!\left(\G_{{\rm npr}}^{\PGSp_4}(p)\backslash \PGSp_4(\Q_\ell)\right)^{\PGSp_4(\Z_\ell)}
\]
have density zero as $p \to \infty$. More concretely, among the forms occurring there, those which do not satisfy the Ramanujan conjecture form a minority.
Since the notation involved is rather technical, we refer the reader to Section~\ref{MCSX}
for precise definitions.

\begin{thm}\label{mainSX}{\rm (Sarnak--Xue type density hypothesis, Theorem~\ref{SX})}
For any real number $q>0$, it holds that 
\[
\lim_{p\to\infty}\frac{M(\Pi_{{\rm sph}},\G_{{\rm npr}}^{\PGSp_4}(p),q)}
{\dim\!\left(M(U_{{\rm npr}})\right)^{\frac{2}{q}}}
=0
\]
\end{thm}

This paper is organized as follows.
In Section~\ref{RIGMBF}, we give a detailed definition of Moret--Bailly families and
$(\ell,\ell)$-isogenies.
An adelic description of the isomorphism classes of Moret--Bailly families is then
presented in Section~\ref{AdelicDesc}.
One of the main technical difficulties is to relate the stabilizer of a non-principal
lattice to the unit groups of the endomorphism rings of Moret--Bailly families.

In Section~\ref{AMFs}, we discuss the connection between algebraic modular forms and
Siegel paramodular forms of degree~$2$ via the Jacquet--Langlands correspondence
proved by van Hoften~\cite{vH}.
We then translate results on Siegel paramodular forms into statements about
$\ell^2(\Gc^{{\rm MB}}(\ell,p))$ using the theory of algebraic modular forms and the theory of 
Bruhat-Tits buildings. 
Combining these results with the local representation theory of $\PGSp_4(\Q_\ell)$,
we establish a Sarnak-Xue type hypothesis in Section~\ref{MCSX}.
This result will be applied to the study of cut-off phenomena, which is left for
future work of the second author with a collaborator.

\subsection*{Acknowledgements}
The first author would like to thank Professor T. Ibukiyama for his valuable comments.
This article grew out of discussions following a talk by the second author
at the University of Tokyo in November 2023.
We thank the University of Tokyo for their hospitality.

\subsection{Notations}\label{notation}
For a set $X$, the cardinality is denoted by $|X|$.
Throughout the paper, we use the Landau asymptotic notations: for positive real-valued functions $f(n)$ and $g(n)$ for integers $n$,
we denote 
%by $f(n)=o(g(n))$ if $f(n)/g(n) \to 0$ as $n \to \infty$,
by $f(n)=O(g(n))$ if there exists a positive constant $C>0$ such that $f(n) \le C g(n)$ for all large enough $n$.

Let $I_m$ be the identity matrix of size $m$. 
Put 
$
J=
\begin{pmatrix}
0 & I_m \\
-I_m & 0
\end{pmatrix}
.
$
We define the symplectic similitude group scheme $GSp_{2m}$ over $\Z$ 
by 
$$\GSp_{2m}(R):=\{M\in M_{2m}(R)\ |\ {}^t M J M=\nu(M)J,\ \text{for some }\nu(M)\in R^\times\}
$$
for each commutative ring $R$ and we call $\nu(M)$ the similitude of $M$. 
The similitude defines a homomorphism 
$\nu:GSp_{2m}\lra GL_1$ of group schemes over $\Z$. We define  
$Sp_{2m}:={\rm Ker}(\nu)$ which is called the symplectic group of rank $m$. 
The similitude splits and in fact it is given by 
$
a\mapsto\diag(I_m,a I_m)
$.
%$a\mapsto \diag(I_m,a I_m)$. 
It follows from this that $G=GSp_{2m}\simeq Sp_{2m} \rtimes GL_1$. 
We are mainly interested in $G=GSp_4$. 

For any algebraic group $H$ over  a field, we denote by $Z_{H}$ 
the center of $H$.

\section{Relative isogeny graphs of Moret-Bailly families}\label{RIGMBF}
In this section, we introduce Moret-Bailly families \cite{MB} and relative isogenies
between them. Using these notions, we construct relative isogeny graphs.
To this end, we first review the relevant results on genus theory, divisors
on the superspecial abelian varieties and their interrelations.
We also provide a survey of the loci of superspecial and supersingular abelian varieties 
in the moduli space of principally polarized abelian varieties (we refer to \cite{Oort}, \cite{LO},\cite{KO},\cite{KO2}). 

\subsection{The genus theory}
In this subsection, we recall the genus theory of quaternion algebra
(for details, see Shimura \cite{Shimura} and Hashimoto-Ibukiyama \cite{HI}).

Let $p$ be a prime number, and
$B$ be a definite quaternion algebra over the field of rational numbers $\Q$
with discriminant $p$. Denote by  $~\bar{}~$  the canonical involution of $B$. 
Let $\mathcal{O}$ be a maximal order in $B$.

We consider the left $B$-vector space $B^n$ (consisting of row vectors) equipped with
a definite quaternionic Hermitian form. It is well known that, after a suitable change of basis, such a Hermitian form 
can be written as
$$
     h(x, y) = \sum_{i= 1}^{n} x_i\bar{y}_i\quad (x=(x_1, \ldots, x_n), y=(y_1,\ldots, y_n)\in B^n).
$$
For each place $v$ of $\Q$, 
let $\Q_v$ denote the completion of $\Q$ at $v$.  We then set
$$
        B_v = B \otimes_\Q\Q_v,\ 
        \mathcal{O}_v = \mathcal{O} \otimes_\Z\Z_v
$$
For a commutative ring $R$, we extend the conjugation on $\mathcal{O}$ to 
$\calO \otimes_\Z R$ by $\overline{x\otimes r}:=\overline{x}\otimes r$ for each $x\in \mathcal{O}$ and 
$r\in R$. Further, for each $\g=(\g_{ij})\in M_n(\mathcal{O}\otimes_\Z R)$,  we define $\overline{\g}:=
(\overline{\g}_{ij})$. 
We define the algebraic group $G_n$ over $\Z$ which represents the following functor  
from the category of rings to the category of sets: 
$$\underline{G}_n:(Rings)\lra (Sets),\ R\mapsto\underline{G}_n(R):=
\{\g\in M_n(\mathcal{O}\otimes_\Z R)\ |\ \g\cdot \overline{\g}^t=\nu(\g)I_n\ \text{for some $\nu(\g)\in R^\times$} \}$$
where $I_n$ stands for the identity matrix of size $n$. 
The algebraic group scheme $G_n$ is called the algebraic group scheme of similitudes for 
$(\mathcal{O}^n, h)$. The character $\nu:G_n\lra GL_1,\ g\mapsto \nu(g)$ is called the similitude 
character. We sometimes write $G_n=GUSp_n$ as in \cite{ATY}. 
%(resp. $(B_v^n, h)$) is defined by
%$$
%\begin{array}{rl}
%G_n \quad  &= \{g \in M_n(B) \mid g\bar{g}^{t} = \lambda (g)I_n , \lambda (g) \in \Q^{*}\}\\
%(\mbox{resp.}~G_{n,v} &= \{g \in M_n(B_v) \mid g\bar{g}^{t} = \lambda (g)I_n , \lambda (g) \in \Q_v^{*}\}).
%\end{array}
%$$
%Here, $I_n$ denotes the identity matrix, and $\bar{g}^{t}$ is the tranpose of 
%the matrix $\bar{g}$.

\begin{Def}
Let $L \subset B^n$ be a left ${\mathcal O}$-module. If $L$ is 
also a $\Z$-lattice, then we call $L$ a left ${\mathcal O}$-lattice. 
\end{Def}

\begin{Def}
Let $L_1, L_2 \subset B^n$ be left ${\mathcal O}$-lattices. We say
$L_1$ is globally equivalent (resp. locally equivalent at a prime $q$) to $L_2$
if there exists $g \in G_n(\Q)$ $($resp. $g \in G(\Q_q)$$)$ such that 
$$
L_1g = L_2 \quad (resp. (L_1 \otimes_\Z\Z_{q})g = (L_2 \otimes_\Z\Z_{q})).
$$
In this case, we write $L_1 \sim L_2$.
\end{Def}

\begin{Def}
For a left ${\mathcal O}$-lattice $L$, the two-sided ideal generated by
$\{h(x,y) \mid x, y \in L\}$ is called the norm of $L$.
A maximal element among left ${\mathcal O}$-lattices
with the same norm is called a maximal left ${\mathcal O}$-lattice.
\end{Def}

Let ${\mathcal L}_n({\mathcal O})$ denote the set of maximal ${\mathcal O}$-lattices.
The following lemma is due to Shimura \cite[Propositions 2.10 and 3.5]{Shimura}
(also see \cite[(II) Section 1]{HI}).

\begin{lem}\label{Shimura} 
\begin{itemize}
\item[$({\rm i})$] If $B_q = M_2(\Q_q)$ for $q \neq p$, then every maximal 
${\mathcal O}_q$-lattice in $B_q^n$ is locally eqivalent to ${\mathcal O}_q^n$.
\item[$({\rm ii})$] If $B_p$ is a division algebra, then every maximal 
${\mathcal O}_p$-lattice in $B_p^n$ is locally eqivalent to either ${\mathcal O}_p^n$
or $N_p$.
\end{itemize}
\end{lem}
Here, $N_p$ is defined as follows. Let $r = [\frac{n}{2}]$, the integral part
of $\frac{n}{2}$, and let $\pi$ be a prime element of ${\mathcal O}_p$.
Define a matrix $\xi \in GL_n(B_p)$ satisfying
$$
    \xi \bar{\xi}^{t} =
    \left(
\begin{array}{ccccc}
    0 &  0 &  \ldots  & 0 &1 \\
    0 &  0 &  \ldots  & 1 &0 \\
      &    &   \ldots   &   &   \\
    0 &  1 &  \ldots  & 0 &0 \\
    1 &  0 &  \ldots  & 0 &0 \\
\end{array}
\right).
$$
Then set
\begin{equation}\label{Np}
   N_p = {\mathcal O}_p^n 
\left(
 \begin{array}{cc}
     I_r &  0  \\
    0 &  \pi I_{n-r} 
\end{array}
\right)
\xi.
\end{equation}
Using this notation, define
$$
\begin{array}{l}
   {\mathcal L}_n(p, 1) = \{\mbox{maximal}~ {\mathcal O}\mbox{-lattice}~L \subset B^n
   \vert L_q \sim {\mathcal O}_q^n~\mbox{for all primes}~ q  \}\\
   {\mathcal L}_n(1, p) = \{\mbox{maximal}~ {\mathcal O}\mbox{-lattice}~L \subset B^n
   \vert L_q \sim {\mathcal O}_q^n~~\mbox{for all}~q \neq p 
   ~\mbox{and}~L_p \sim N_p \}.
\end{array}
$$
Then we have
$$
     {\mathcal L}_n({\mathcal O}) = {\mathcal L}_n(p, 1)\cup {\mathcal L}_n(1, p).
$$
We call ${\mathcal L}_n(p, 1)$ the principal genus and ${\mathcal L}_n(1, p)$
 the non-principal genus. Define the class numbers as:
$$
\begin{array}{l}
H_n(p, 1) = \vert {\mathcal L}_n(p, 1)/\mbox{global equivalence}\vert.  \\
H_n(1, p) = \vert {\mathcal L}_n(1, p)/\mbox{global equivalence}\vert. 
\end{array}
$$     
We call $H_n(p, 1)$ the class number of principal genus and 
$H_n(1, p)$ the class number of non-principal genus.

\subsection{Abelian varieties}\label{1.2}
Let $k$ be an algebraically closed field of charactersitic $p$,
and $A$ be an abelian variety of dimension $g$.
For a positive integer $n$, $[n]_A : A \longrightarrow A$
is the morphism defined by $x \mapsto nx$  for $x \in A$.
We set $A[n]= {\rm Ker}~ [n]_A$.
We denote by $\widehat{A} = {\rm Pic}^0(A)$ the dual abelian variety of $A$.
Then, for an invertible sheaf (or a divisor) $\mL$ on $A$, we have a homomorphism
$$
\begin{array}{rccc}
   \varphi_\mL :&  A  & \longrightarrow & {\widehat{A}} \\
         &     x & \mapsto    & T_x^*\mL - \mL.
\end{array}
$$
Here, $T_x$ is the translation by $x \in A$.
We set $K(\mL) = {\rm Ker}~\varphi_\mL$.

An abelian variety $A$ of dimension $n\geq 2$ is
said to be supersingular if $A$ is isogenous to a  product of $n$ supersingular elliptic curves,
and $A$ is said to be superspecial if $A$ is isomorphic to a  product of $n$ supersingular elliptic curves.
These notions don't depend on the choice of supersingular elliptic curves. Moreover, 
any two products of $n$ supersingular elliptic curves are isomorphic 
(for instance,  Shioda \cite[Theorem 3.5]{S}).
Therefore, if you take a supersingular elliptic curve $E$, the superspecial abelian variety of dimension $n$ 
($n \geq 2$) is uniquely given by $E^n$.

The following proposition is probably well known, but  we could not find a suitable reference,
so we include a proof.

\begin{prop}\label{rational-morphism}
Let $C$ be a non-singular algebraic curve. Let $\varphi : X \longrightarrow C$ and 
$\psi : Y \longrightarrow C$ be abelian schemes, and $f : X \longrightarrow Y$ be
a relative rational map over $C$.  Then, $f$ is a morphism.
\end{prop}

%To prove this proposition, we prepare a lemma.

%\begin{lem}
%Let $C$ be a non-singular algebraic curve. Let $\varphi : X \longrightarrow C$ and 
%$\psi : Y \longrightarrow C$ be proper algebraic varieties, and $f : X \longrightarrow Y$ be
%a relative rational map over $C$. Assume that $f$ induces a morphism on any fiber of
%$f : X \longrightarrow Y$. Then, $f$ is a morphism.
%\end{lem}
\begin{proof}
Let $V$ be the indeterminacy locus of the rational map $f$. Then by the general theory
of rational maps, we have $\dim V \leq \dim X -2$. Since each fiber of 
$\varphi : X \longrightarrow C$ has dimension $\dim X -1$, it follows that $V$ does not contain any fiber.
Hence, $f$ induces a rational map on every fiber. Moreover, since any rational map
from a nonsingular variety to an abelian variety is a morphim (cf. Lang \cite[Section II, Theorem 2]{L}),
we see that $f$ induces a morphim between a fiber of $\varphi$ and a fiber of $\psi$.

We consider the following commutative diagram:
%$$
%\begin{array}{ccc}
%X   & \stackrel{{\rm pr}_1}{\longleftarrow} & X\times_{C}Y \\
%\varphi \downarrow \quad \quad  &    & \quad \quad\downarrow {\rm pr}_2\\
%C & \stackrel{\psi}{\longleftarrow} & Y.
%\end{array}
%$$
\[
\xymatrix{
X\ar[d]_\varphi & X \times_C Y \ar[l]_{{\rm pr}_1\hspace{2mm}} \ar[d]^{{\rm pr}_2} \\
C  & Y \ar[l]_\psi
}
\]

We set $U =  X \setminus V$.
Then $U$ is a Zariski open subset of $X$. We denote by $f_{U}$ the restriction of $f$ to $U$.
Thus $f_U : U \longrightarrow Y$ is a relative morphism over $C$. 
Let $\Gamma_U$ be the graph of $f_U$, and let $\Gamma$ denote the Zariski closure of $\Gamma_U$
in $X\times_{C}Y$. Denote by $\gamma$ the restriction morphism of ${\rm pr}_1$ to $\Gamma$.
For $a \in C$ we consider the fiber $X_a=\varphi^{-1}(a)$. Then, as we saw above, the restriction 
$f\vert_{X_a}: X_a \longrightarrow Y_a$ is a morphism. Moreover, since we have 
${\rm pr}_2(\gamma^{-1}(X_a)) \subset Y_b$, we see that ${\rm pr}_1 : \gamma^{-1}(X_a) \longrightarrow X_a$ 
is an isomorphism.
Therefore, $f\vert_{X_a}: X_a \longrightarrow Y$ (not only to $Y_a$) is a morphism.
This shows that $f$ is defined at every point on $X$. Hence, we conclude that $f$ is a morphism.
\end{proof}

Let $C$ be a nonsingular curve over $k$, and let $\varphi : X \longrightarrow C$  be an
 abelian scheme. Denote by $X_{k(C)}$ the generic fiber of $\varphi$,
 and by ${\rm End}(X/C)$ the ring of relative endomorphisms of $\varphi : X \longrightarrow C$.
For an open set $U \subset C$, we consider the restricted morphism $\varphi\vert_{\varphi^{-1}(U)}:\varphi^{-1}(U)\longrightarrow U$. We denote by ${\rm End}(\varphi^{-1}(U)/U)$ 
the set of endomorphisms of $\varphi^{-1}(U)$ over $U$, and simply write it as ${\rm End}(X/U)$. 
%For each geometric point $s$ of $C$, we denote by $\calX_s$ the special fiber of $\varphi$ at $s$.
 
\begin{cor}\label{endo}
There is a natural isomorphism ${\rm End}(X/C) \cong {\rm End}(X_{k(C)})$. 
In particular, for any non-empty open variety $U$ of $C$, we have  
${\rm End}(X/C) \cong {\rm End}(X/U)$. 
 \end{cor}
\begin{proof}
We clearly have a natural injection
$$
     {\rm End}(X/C) \hookrightarrow {\rm End}(X_{k(C)}).
$$
Take an element $f \in {\rm End}(X_{k(C)})$. Then $f$ defines a relative rational map
from $X$ to $X$. By Proposition \ref{rational-morphism} this map is in fact a morphism.
Hence, $f$ belongs to ${\rm End}(X/C)$. 
\end{proof}

\subsection{Moduli and genus} 
Let $n\ge 1$ and $p$ be a prime.  
Let ${\mathcal A}_{n, 1}$ denote the moduli space of principally 
polarized abelian varieties over $\bF_p$ of dimension $n$. Let ${\mathcal S}_{n,\bF_p}$
be the supersingular locus, and  ${\mathcal SS}_{n,\bF_p} \subset {\mathcal S}_{n,\bF_p}$ 
the superspecial locus in ${\mathcal A}_{n, 1}$.  
Then, 
${\mathcal S}_{n,\bF_p}$ is a union of subvarieties of dimension $[\frac{n}{4}]$ (cf. \cite[Section 4.9 Theorem]{LO})
and ${\mathcal SS}_{n,\bF_p}$ consists of finitely many  points. 

\begin{thm}[Ibukiyama-Katsua-Oort \cite{IKO}]\label{IKO}
   $\vert {\mathcal SS}_{n,\bF_p}\vert  = H_n(p, 1)$.
\end{thm}

\begin{thm}[$n=2, 3$ by Katsura-Oort\cite{KO},\cite{KO2}, $n \geq 4$ by Oort-Li\cite{LO}]\label{KLO}
\begin{itemize}
\item[$({\rm i})$]If $n$ is odd, the number of irreducible components of 
${\mathcal S}_{n,\bF_p}$ is equal to the class number $H_n(p, 1)$ of the principal genus.
\item[$({\rm ii})$]If $n$ is even, the number of irreducible components of 
${\mathcal S}_{n,\bF_p}$ is equal to the class number $H_n(1, p)$ of the non-principal genus .
\end{itemize}
\end{thm}

Theorem \ref{IKO} was used in \cite{ATY} to construct the isogeny graph of
principally polarized superspecial abelian varieties. 
In this paper, we use Theorem \ref{KLO}
to construct the relative isogeny graph in the case $n= 2$.
Below, we provide the explicit formula of $H_2(1, p)$ (cf. \cite[(II), p.696 THEOREM]{HI}):
\begin{equation}\label{classnumber}
\begin{array}{rl}
H_2(1, p) = &\ds\frac{p^2 -1}{2880} + \frac{1}{64}(p+1)\Big\{1 - \Big(\frac{-1}{p}\Big)\Big\} + 
\frac{5}{192}(p-1)\Big\{1 + \Big(\frac{-1}{p}\Big)\Big\}       \\
          & +\ds\frac{1}{72}(p+1)\Big\{1 - \Big(\frac{-3}{p}\Big)\Big\} + 
\frac{1}{36}(p-1)\Big\{1 + \Big(\frac{-3}{p}\Big)\Big\} \\
 &
+ \left\{
 \begin{array}{l}
 \frac{2}{5} \cdots p \equiv 2~\mbox{or}~ 3~ (\mbox{mod}~5)\\
 0 \cdots p \equiv 1~\mbox{or}~ 4 ~(\mbox{mod}~5)
 \end{array}
 \right.
+ \left\{
 \begin{array}{l}
 \frac{1}{4} \cdots p \equiv 3~\mbox{or}~ 5~ (\mbox{mod}~8)\\
 0 \cdots p \equiv 1~\mbox{or}~ 7 ~(\mbox{mod}~8)
 \end{array}
 \right. \\
 &+ \left\{
  \begin{array}{l}
 \frac{1}{6} \cdots p \equiv 5~(\mbox{mod}~12)\\
 0 \cdots p \equiv 1, 7~\mbox{or}~ 11 ~(\mbox{mod}~12)
 \end{array}
 \right.
\end{array}
\end{equation}
if $p\ge 7$ and $H_2(1,p)=1$ for $2\le p\le 5$. 
We list $p$ and $H_2(1,p)$ with  $H_2(1,p)\le 10$ as below:
\begin{table}[htbp]
\centering 
\begin{tabular}{|c|c|c|c|c|c|c|c|c|c|c|} \hline
$H_2(1,p)$ & $1$ & $2$ & $3$  & $4$ & $5$ & $6$ & $7$ &  $8$ & $9$ & $10$ \\ \hline
$p$ & $p\le 11$ & $13, 17, 19, 23$ & $29, 31$  & $41, 47$ & $37, 43, 59$ & $53, 71$ & 
none &  $61, 67, 83$ & $79$ & $73, 89$ \\ \hline
\end{tabular}
\medskip
\caption{$H_2(1,p)\le 10$.}
\label{table1}
\end{table}

\subsection{IKO correspondence}
For the remainder of this section,
we fix a supersingular elliptic curve $E$ over $\bbF_p$ with
identity element $O_{E}$.
We also regard $E$ as an elliptic curve over $k=\overline{\bbF}_p$.
Set ${\mathcal O} = {\rm End}(E)$ and $B = {\rm End}^{0}(E)={\rm End}(E)\otimes_\Z \Q$.
Then $B$ is a quaternion division algebra over the rational number field $\Q$
with discriminant $p$, and ${\mathcal O}$ is a maximal order of $B$ 
(cf. Mumford \cite{M}, Section 22). 
Consider the superspecial abelian surface $A = E_{1} \times E_{2}$
with $E_{1} = E_{2} = E$, and take a divisor 
$$
  \Theta = E_{1}\times\{O_{E_2}\} + \{O_{E_1}\} \times E_{2}. 
$$
Then $\Theta$ is a principal polarization on $A$.
For simplicity, we also denote $E_{1}\times\{O_{E_2}\}$  
(resp.  $\{O_{E_1}\} \times E_{2}$) simply by $E_{1}$ (resp. by $E_{2}$)
for the sake of simplicity.

There is a natural identification of ${\rm End}(A)$ 
with the matrix ring ${\rm M_{2}}({\mathcal O})$:
$$
               {\rm End}(A) = {\rm M_{2}}({\mathcal O}).
$$
Explicitly, the action of $\left(
\begin{array}{cc}
\alpha  & \beta \\
\gamma & \delta
\end{array}
\right) \in {\rm M_{2}}({\mathcal O})$ on $A = E_1 \times E_2$ is given by
$$
\begin{array}{rccc}
\left(
\begin{array}{cc}
\alpha  & \beta \\
\gamma & \delta
\end{array}
\right) : &  A = E\times E & \longrightarrow & A = E\times E \\
    & (x, y)   &\mapsto & (\alpha (x) + \beta (y), \gamma (x) + \delta (y)).
\end{array}
$$

From now on, by a divisor $L$ we often mean its divisor class
in ${\rm NS}(A)$ whenever no confusion arises.
Since the map $\varphi_{\Theta} : A \longrightarrow {\rm Pic}^{0}(A)$ is an isomorphism,
for a divisor $L$ we obtain an endomorphism 
$\varphi_{\Theta}^{-1}\circ \varphi_L \in {\rm End}(A) \cong {\rm M}_{2}({\mathcal O})$.
Define  
$$
     {\rm H} = \Bigg\{
\left(
\begin{array}{cc}
\alpha  & \beta \\
\gamma & \delta
\end{array}
\right) \in {\rm M}_{2}({\mathcal O})
~\Bigg|
~\alpha, \delta \in \Z,~\gamma, \beta \in {\mathcal O},~\gamma = \bar{\beta}
\Bigg\} \subset {\rm M}_{2}({\mathcal O}).
$$
Using the notation in Subsection \ref{1.2}, we have the following correspondence.
\begin{thm}[Mumford \cite{M}, Ibukiyama, Katsura and Oort \cite{IKO}
and Katsura \cite{K}]\label{intersection}
The homomorphism
$$
\begin{array}{cccc}
j : & {\rm NS}(A) & \longrightarrow  & {\rm H} \\
   & L  &\mapsto  & \varphi_{\Theta}^{-1}\circ\varphi_{L}
\end{array}
$$
is bijective. Moreover,
for $L_{1}, L_{2} \in {\rm NS}(A)$ with
$$
  j(L_{1}) =
\left(
\begin{array}{cc}
\alpha_{1}  & \beta_{1} \\
\gamma_{1} & \delta_{1}
\end{array}
\right),~
  j(L_{2}) =
\left(
\begin{array}{cc}
\alpha_{2}  & \beta_{2} \\
\gamma_{2} & \delta_{2}
\end{array}
\right),
$$
the intersection number is given by
$$
L_{1}\cdot L_{2} = \alpha_{2}\delta_{1} + \alpha_{1}\delta_{2} - \gamma_{1}\beta_{2}
-\gamma_{2}\beta_{1}.
$$
For an endomorphism $r \in {\rm End}(A)$, we have
$$
         j(r^*L) = \bar{r}^t j(L) r 
$$
as an element of  ${\rm End}(A) = {\rm M_{2}}({\mathcal O})$.
\end{thm}

\begin{rmk}
In \cite[Lemma 2.7]{K}, the first author assumed that $r$ is an automorphism and showed that 
$j(r^*L) = \bar{r}^t j(L)r$. However, the same proof remains valid when $r$ is merely an endomorphism.
\end{rmk}
 
\subsection{The Moret-Bailly family}\label{family}
In this subsection, we summarize the construction of Moret-Bailly families
of principally polarized supersingular abelian surfaces (Moret-Bailly \cite[Sections 1 and 2]{MB}).
Let $k$ be an algebraically closed field of charactersitic $p\geq 3$,
and $E$ be a supersingular elliptic curve defined over $k$.
We set $A = E \times E$, which is the unique superspecial abelian surface.
Let $\alpha_p = {\rm Spec}~ k[\epsilon]/(\epsilon^p)$ be the local-local group scheme
of rank $p$.
Then we have $\alpha_p \times \alpha_p \subset A$. Let ${\mathbf P}^1={\mathbf P}^1_k$ 
be the projective line over $k$. 
Gluing the families of the inclusions
\begin{equation}\label{iota}
\iota_t : \alpha_p\stackrel{(t, 1)}{\longrightarrow} \alpha_p \times \alpha_p \subset A\ 
(t\in k),\ 
\iota_s : \alpha_p\stackrel{(1, s)}{\longrightarrow} \alpha_p \times \alpha_p \subset A,\ 
(s\in k), 
\end{equation} 
by $st=1$, we have a relative subgroup scheme 
\begin{equation}\label{Q}
\mathcal{Q} \subset (\alpha_p \times \alpha_p) \times {\mathbf P}^1
\end{equation}
over ${\bf P^1}$ with parameter $t$ ($t \in k$ or $t = \infty$).
We have a sequence 
%$$
%\begin{array}{ccccccccc}
%0 &\rightarrow & Q & \longrightarrow & A \times {\mathbf P}^1 & \stackrel{\pi}{\longrightarrow} 
%& {\mathcal X} & \rightarrow & 0\\
%  &       &    &  {\rm pr}_1 \swarrow  \quad \quad & \downarrow  {\rm pr}_2 &\quad  \swarrow q & & & \\
%  &       & A   &           &{\mathbf P}^1 &            &  &  &
%\end{array}
%$$
\[
\xymatrix{
0 \ar[r] & \mathcal Q \ar[r] & A \times {\mathbf P}^1 \ar[r]^{\pi} \ar[dl]_{{\rm pr}_1} \ar[d]^{{\rm pr}_2} 
  & {\mathcal X}\ar[dl]_{q} \ar[r] & 0 \\
&A & {\mathbf P}^1  & & 
}
\]

with
${\mathcal X} \cong (A \times {\mathbf P}^1)/\mathcal{Q}$.

Now, we take an invertible sheaf  $\mL$ on $ E\times E = A$ which satisfies
the following two conditions:
\begin{itemize}
\item[$({\rm i})$] $\mL$ is symmetric.
\item[$({\rm ii})$] $K(\mL)\cong \alpha_p \times \alpha_p$. 
\end{itemize}
Let $\iota$ be the inversion of $A$. Then, ``symmetric'' means
that $\iota^\ast\mL \cong \mL$.
Using the descent theory by Mumford (\cite
[Corollary of Theorem 2 in Section 23]{M}), for the projection
$$
      \pi_t :  A \longrightarrow A/\iota_t(\alpha)
$$
there exists an invertible sheaf  $\mM_t$ on $A/\iota_t(\alpha)$ 
such that $\mL \cong \pi_t^*\mM_t$.
Considering relatively this situation (cf. Moret-Bailly \cite[Subsection 2.1]{MB}), 
we see there exists an invertible sheaf $\mM$ on ${\mathcal X}$
such that 
\begin{equation}\label{prL}
{\rm pr}_1^*\mL \cong \pi^*\mM.
\end{equation}
Then, by Moret-Bailly \cite[Subsection 2.1]{MB}, there exist an effective relative divisor $D$ 
on ${\mathcal X}$ such that 
$$
{\mathcal O}_{\mathcal X}(D) \cong M \otimes q^*{\mathcal O}_{{\mathbf P}^1}\Big(\frac{p-1}{2}\Big).
$$
The effective relative divisor $D$ is a family of principal polarizations on the fibers of 
$q : {\mathcal X} \longrightarrow {\mathbf P}^1$. 
The pair $({\mathcal X}, D)$ or $({\mathcal X}, \mM)$ is called a Moret-Bailly family of principally polarized supersingular abelian surfaces.

\begin{rmk}
Even in the case of $p=2$, one can construct a similar Moret-Bailly family by a different method
(cf. Moret-Bailly \cite{MB1}).
\end{rmk}

We give here a remark on the ${\ell}$-adic Tate module $T_{\ell}({\mathcal X}/{\mathbf P}^1)$. 
Let $n$ be a positive integer.
Let $X$ denote the generic fiber of $q: {\mathcal X}\longrightarrow {\mathbf P}^1$, and let
$[n]_{{\mathcal X}/{\mathbf P}^1}$  denote the relative multiplication-by-$n$ morphism 
associated with 
$q : {\mathcal X} \longrightarrow {\mathbf P}^1$. Assume that  $n$ is prime to $p$. Then 
$\mathrm{Ker} [n]_{{\mathcal X}/{\mathbf P}^1}$ is \'etale over ${\mathbf P}^1$. 
Since ${\mathbf P}^1$ is simply connected,
$\mathrm{Ker} [n]_{{\mathcal X}/{\mathbf P}^1}$ splits as a disjoint union of $n^4$ copies 
of ${\mathbf P}^1$.
We denote by ${\mathcal X}/{\mathbf P}^1[n]$ 
(resp. $(E\times E \times {\mathbf P}^1)/{\mathbf P}^1[n]$) the group of 
$n$-torsion sections. Since $\pi$ is purely inseparable, we have isomorphisms
$$
    (E\times E \times {\mathbf P}^1)/{\mathbf P}^1[n] \cong {\mathcal X}/{\mathbf P}^1[n] \cong ({\mathbf Z}/n{\mathbf Z})^4.
$$
Therefore, we have natural isomorphisms
$$
T_{\ell}(E\times E)\cong T_{\ell}((E\times E \times {\mathbf P}^1)/{\mathbf P}^1)\cong 
T_{\ell}({\mathcal X}/{\mathbf P}^1)\cong   T_{\ell}(X).
$$
For the endomorphism algebra, Corollary \ref{endo} yields
$$
   {\rm End}(X) \cong {\rm End}({\mathcal X}/{\mathbf P}^1) \subset {\rm End}(E \times E).
$$

\subsection{Construction of relative isogenies}\label{reliso}
Let ${\mathcal D}$ be the set of all invertible sheaves on $A$ which satisfy the conditions
(i) and (ii) in Subsection \ref{family}.
We consider generalized Richelot isogenies in ${\mathcal D}$.
We take $\mL \in {\mathcal D}$ and a prime number $\ell$ which is prime to $p$. 
%Here, we consider $L$ a divisor on $A$. 
There exist subgroup schemes $K_1(\mL)$ and $K_2(\mL)$ such that
\begin{equation}\label{K1K2}
{\rm Ker}~ \varphi_{\mL^{\otimes \ell}}\cong K_1(\mL) \times K_2(\mL)
\end{equation}
with $K_1(\mL) \cong \alpha_p \times \alpha_p$ and $K_2(\mL) \cong (\Z/\ell \Z)^4$.
The group scheme $K(\mL^{\otimes \ell}) = {\rm Ker}~ \varphi_{\mL^{\otimes \ell}}$ has the skew-symmetric bilinear 
homomorphism associated to the commutator in the theta group ${\mathcal G}(\mL^{\otimes \ell})$
(cf. Mumford \cite[Section 23]{M}). Let $G$ be a maximal isotropic subgroup of $K_2(\mL)$
with respect to this skew symmetric bilinear homomorphism.
Then, $G \cong (\Z/\ell \Z)^2$. We consider the quotient surface 
$A/G$ :
$$
    \phi : A \longrightarrow A/G.
$$  
By the descent theory of Mumford \cite[Corollary of Theorem 2 in Section 23]{M},
there exists a divisor $\mL'$ on $A/G$ such that $\mL^{\otimes \ell} \sim \phi^*\mL'$ and
a commutative diagram:
%$$
%\begin{array}{ccc}
%   A   & \stackrel{\varphi_{L^{\otimes \ell}}}{\longrightarrow} & \widehat{A}\\
%\phi \downarrow &      & \uparrow \widehat{\phi} \\
%A/G    & \stackrel{\varphi_{L'}}{\longrightarrow} & \widehat{(A/G)}.
%\end{array}
%$$
\[
\xymatrix{
A \ar[r]^{\varphi_{\mL^{\otimes \ell}}} \ar[d]_\phi & \widehat{A} \\
A/G \ar[r]^{\varphi_{\mL'}} & \widehat{(A/G)}\ar[u]_{\widehat{\phi}} 
}
\]
Observe that $\mL'$ is determined, up to algebraic equivalence, uniquely by $\mL^{\otimes \ell}$.
Since ${\rm Ker}~\phi \cong (\Z/\ell \Z)^2$ and ${\rm Ker}~\widehat{\phi}$
is dual to ${\rm Ker}~\phi$, we see ${\rm Ker}~\widehat{\phi} \cong (\Z/\ell \Z)^2$.
Therefore, we have ${\rm Ker}~\varphi_{\mL'}\cong \alpha_p \times \alpha_p$.
We may assume $\mL'$ is symmetric by a suitable translation and $\mL^{\otimes \ell} \sim \phi^*\mL'$.
Therefore, we see $A/G \cong E \times E = A$ (cf. Oort \cite[Theorem 2]{O}) and 
$\mL' \in {\mathcal D}$. Since the degree of $\phi$ is not divisible by $p$, we also see 
$\phi(\alpha_p \times \alpha_p) \cong \alpha_p \times \alpha_p$
and $\phi(\iota_t(\alpha_p)) \cong \alpha_p$.
Therefore, we have a commutative diagram:
%$$
%\begin{array}{cccc}
%          & A & \stackrel{\pi_t}{\longrightarrow} & A/\iota_t(\alpha_p)\\
%        & \phi  \downarrow &             & \downarrow \phi'    \\
%      A \cong    & A/G   &  \stackrel{\rho_t}{\longrightarrow} &(A/G)/\phi(\iota_t(\alpha_p))
%\end{array}
%$$
\[
\xymatrix{
A \ar[r]^{\pi_t \hspace{5mm}} \ar[d]_\phi & A/\iota_t(\alpha_p)\ar[d]^{\phi'}  \\
A\simeq A/G \ar[r]^{\rho_t\hspace{5mm}} & (A/G)/\phi(\iota_t(\alpha_p))
}
\]
Here, $\phi'$ is the homomorphism induced by $\phi$.

\begin{lem}\label{RichelotIsogeny}
The homomorphism $\phi': A/\iota_t(\alpha_p) \longrightarrow (A/G)/\phi(\iota_t(\alpha_p))$
is a Richelot isogeny. Namely, if we denote by $\mM'_t$ the invertible sheaf on 
$(A/G)/\phi(\iota_t(\alpha_p))$ given by $\mL' = \rho_t^*\mM'_t$, then we have
$\mM_t^{\otimes \ell} \cong \phi'^*\mM'_t$.
\end{lem}
\proof{
This lemma follows from the fact that the homomorphism
$\pi_t^* : {\rm Pic}^0(A/\iota_t(\alpha_p)) \longrightarrow {\rm Pic}^0(A)$
is purely inseparable.
}

\begin{Def}\label{DefRI} 
The homomorphism $\phi'$ in the above lemma associated to a maximal isotropic subgroup of $K_2(\mL)\simeq 
(\Z/\ell\Z)^4$ is called  a relative Richelot isogeny or a relative $(\ell, \ell)$-isogeny.
\end{Def}

\subsection{Divisors and non-principal genus}
We take ${\mathcal O}= {\rm End}(E)$ and $B = {\rm End}(E)\otimes_{\Z} \Q$.
For every left ${\mathcal O}$-lattice $M$ in $B^n$, there exists $x\in GL_n(B)$ such that
$M ={\mathcal O}^n x$. 
Denoting by $F$ the Frobenius morphism on the supersingular elliptic curve $E$, we set
$$
\Lambda = 
\left\{
\left(
\begin{array}{cc}
ps   &   r\\
\bar{r} & pt
\end{array}
\right)
~\Bigg|
~s, t \in \Z, s>0, t>0, r \in F{\mathcal O}~\mbox{such that}~ p^2st - r\bar{r} = p
\right\}
\subset {\rm H}
$$
For $g_1, g_2 \in \Lambda$, we say that $g_1$ is equivalent to $g_2$ 
if there exist $\gamma \in GL_2({\mathcal O})$
and $m \in \Q_{>0}$ such that $\bar{\gamma}^t g_1 \gamma = m g_2$.
We denote this equivalent relation by $g_1 \sim g_2$.

The following lemma characterizes the elements of the non-principal genus
(cf. \cite[Lemma 2.6]{IKO}).
\begin{lem}\label{characterization}
A left ${\mathcal O}$-lattice $M= {\mathcal O}^nx$ in $B^n$ with $x \in GL_n(B)$ belongs to 
the non-principal genus ${\mathcal L}_n(1, p)$ if and only if $x\bar{x}^t = m g$ for some 
$m\in \Q_{>0}$ and $g \in \Lambda$.
\end{lem}

Let $M_1 = {\mathcal O}^nx_1$ and $M_2 = {\mathcal O}^nx_2$ $(x_1, x_2 \in GL_n(B))$
be two elements of ${\mathcal L}_n(1, p)$. As in Lemma \ref{characterization}, we have
$$
  x_i\bar{x}_i^t = m_ig_i ~\mbox{for some}~m_i\in \Q_{>0}\mbox{ and}~g_i \in \Lambda~(i=1,2).
$$
Then $M_1$ is globally equivalent to $M_2$ if and only if $g_1\sim g_2$.

Recall that $\mathcal D$ denotes the set of divisor classes in ${\rm NS}(A)$ 
satisfying the conditions (i) and (ii) in Subsection \ref{family}. We sometimes regards its elements as ample divisors of $A$ via the first Chern class map. 
We have an isomorphism
\begin{equation}\label{DLambda}
\begin{array}{rccc}
j\vert_{\mathcal D} : &{\mathcal D} & \cong & \Lambda\\
        &  \mL   & \mapsto  & \varphi_{\Theta}^{-1}\circ \varphi_{\mL}. 
\end{array}
\end{equation}
There is a natural isomorphism ${\rm Aut}(A) \cong \GL_2({\mathcal O})$.  
For $\theta \in {\rm Aut}(A)$, the action on ${\mathcal D}$ is given by 
$$
     \mL \mapsto \theta^{*}\mL ~(\mL \in {\mathcal D})
$$
and this action corresponds to the action of $\gamma \in GL_2({\mathcal O})$ on $\Lambda$,
namely
$$
         g \mapsto \bar{\gamma}^tg\gamma.
$$

Recall that  $A_{2, 1}$ denotes the moduli space of principally polarized abelian surfaces,
and that ${\mathcal S}_{2,\bF_p}$ denotes the locus of supersingular abelian surfaces 
in $A_{2,1}$. When no confusion arises, we sometimes regard ${\mathcal S}_{2,\bF_p}$ as
the set of irreducible components of the locus 
of the supersingular locus in $A_{2,1}$.
We also denote by ${\rm MB}(p)$ the set of isomorphism classes of Moret-Bailly families 
of principally polarized supersingular abelian surfaces. 
For a Moret-Bailly family $q : ({\mathcal X}, \mM) \longrightarrow {\mathbf P}^1$, 
by universality of $A_{2,1}$, we obtain a morphism
$$
      \phi: {\mathbf P}^1 \longrightarrow A_{2,1}.
$$
The image of $\phi$ yields an irreducible component of ${\mathcal S}_{2,\bF_p}$. Conversely, every irreducible
component of ${\mathcal S}_{2,\bF_p}$ arises in this method (cf. \cite[Corollary 2.2]{KO}).
The following theorem provides an explanation of Theorem \ref{KLO} in the case $n=2$.
\begin{thm}\label{1to1}
Under the notation above, there exist the following one-to-one correspondences:
$$
   {\mathcal L}_2(1, p)/\mbox{global equivalence}\longleftrightarrow \Lambda/\sim ~\longleftrightarrow {\mathcal D}/{\rm Aut}(A) \longleftrightarrow {\rm MB}(p) \longleftrightarrow {\mathcal S}_{2,\bF_p}.
$$
\end{thm}
For details on this subsection, see \cite[Section 2.3]{IKO} and \cite[Theorem 2.7 and Corollary 3.8]{KO}.

\subsection{Relative isogeny graphs}
For two  Moret-Bailly families $({\mathcal X}, \mM)$ and
 $({\mathcal X}', \mM')$ which corresponds to $(A, \mL)$ and 
 $(A/G\cong A, \mL')$ for a maximal totally isotropic subgroup $G$ of $K_2(\mL)$, 
 by Lemma \ref{RichelotIsogeny}, we have a relative Richelot isogeny (a relative $(\ell, \ell)$-isogeny) $f$:
%$$
%\begin{array}{ccc}
%     {\mathcal X} &\stackrel{f}{\longrightarrow} & {\mathcal X}'  \\
%    q \downarrow &           &  \downarrow  q' \\
%    {\mathbf P}^1  & \longrightarrow     & {\mathbf P}^1.
%\end{array}
%$$

\[
\xymatrix{
{\mathcal X} \ar[r]^f \ar[dr]_{q}^{\circlearrowright} & {\mathcal X}' \ar[d]^{q'} \\
& {\mathbf P}^1
}
\] 
(see (\ref{K1K2}) for $K_2(\mL)$). The group $G$ is uniquely extended to a relative 
finite \'etale group subscheme $\mathcal{G}$ of $\mathcal X$. Then, we define the 
weight of $f$ by 
\begin{equation}\label{weight}
w(f):=|({\rm RA}(({\mathcal X}, \mM)/ {\mathbf P}^1)/{\rm Stab}_{{\rm RA}(({\mathcal X}, \mM)/ {\mathbf P}^1)}(\mathcal{G}))|
\end{equation}
where ${\rm RA}(({\mathcal X}, \mM)/ {\mathbf P}^1)={\rm Aut}(({\mathcal X}, \mM)/ {\mathbf P}^1)/\{\pm 1\}$ 
is the relative reduced automorphism group of $({\mathcal X}, \mM)$ and 
${\rm Stab}_{{\rm RA}(({\mathcal X}, \mM)/ {\mathbf P}^1)}(\mathcal{G})$ stands for the stabilizer of 
$\mathcal{G}$ 
inside ${\rm RA}(({\mathcal X}, \mM)/ {\mathbf P}^1)$. 

We now define the supersingular relative isogeny graph ${\mathcal G}^{{\rm MB}}(\ell, p)$ of Moret-Bailly families as follows:
\begin{itemize}
\item the vertices are the elements of ${\rm MB}(p)$, 
\item the oriented edges are relative $(\ell, \ell)$-isogenies and   
two edges $f_1,f_2$ starts from $({\mathcal X}, \mM)$ 
associated to maximal totally isotropic subgroups $G_1,G_2$ of $K_2(\mL)$ with unique extensions 
$\mathcal {G}_1,\mathcal{G}_2$ inside ${\mathcal X}$ are identified if 
 $$\mathcal{G}_1=g\mathcal{G}_2$$ for some $g\in {\rm RA}(({\mathcal X}, \mM)/ {\mathbf P}^1)$, 
 and 
 \item each edge $f$ has the weight $w(f)$ which is well-defined under the above identification. 
\end{itemize}

\begin{thm}\label{graphconst} The graph ${\mathcal G}^{{\rm MB}}(\ell, p)$ is a 
finite oriented graph such that 
\begin{enumerate}
\item the number of vertices is equal to the class number $H_2(1, p)$ of non-principal genus of 
the quaternion hermitian form in $B^2$, and 
\item the weighted sum $\ds\sum_{v'\in V,\ f:v\to v'}w(f)$ of the edges outgoing from any fixed vertex $v$ is equal to $N_2(\ell):=(\ell + 1)(\ell^2 + 1)$. 
\end{enumerate}
\end{thm}
\begin{proof}
By Theorem \ref{1to1}, the number of vertices is equal to $H_2(1,p)$. 
Since the number of all maximal isotropic subgroups of $K_2(\mL)\simeq (\Z/\ell\Z)^4$ is 
equal to $(\ell + 1)(\ell^2 + 1)$, the second claim follows from the definition of 
the weight. 
\end{proof}

\begin{rmk}\label{notsimple}
Note that the graph ${\mathcal G}^{{\rm MB}}(\ell, p)$ may have loops, namely, 
may not be simple. 
\end{rmk}

\section{An adelic description of ${\mathcal G}^{{\rm MB}}(\ell, p)$}\label{AdelicDesc}

In this section, we focus on $H:=G_2$. 
Let $\A=\A_\Q$ be the ring of adeles of $\Q$ and $\A_f$ be the finite part of $\A$.
For each prime $\ell\neq p$, there exists an isomorphism  
\begin{equation}\label{Ghitza}
H(\Q_\ell)\simeq \GSp_4(\Q_\ell)
\end{equation} 
satisfying that
\begin{itemize}
\item $H(\Z_\ell)\simeq \GSp_4(\Z_\ell)$, 
\item the similitude is preserved, 
\item $\diag(1,\ell)$ corresponds to $\diag(1,1,\ell,\ell)$, and 
\item the transpose conjugate on $H(\Q_\ell)$ 
corresponds to $g\mapsto J^{-1}\cdot g^t \cdot J$ on $\GSp_4(\Q_\ell)$
\end{itemize}
(cf.\ Lemma 4 of \cite{Ghitza}).  

For an $\calO$-lattice $L$ and each rational prime $\ell$, put $K_{\ell}(L):=
\{\gamma_{\ell}\in H(\Q_{\ell})\ |\ (L\otimes_\Z\Z_{\ell})\gamma_{\ell}=L\otimes_\Z\Z_{\ell}\}$ which is an open compact subgroup of $H(\Q_{\ell})$. Then $K(L):=\ds\prod_{\ell}K_{\ell}(L)$ makes up an 
open compact subgroup of $H(\A_f)$. In particular, $K_\ell(L)=H(\Z_\ell)$ for 
all but finitely many $\ell$ different from $p$.  

For each element 
$\gamma=(\gamma_{\ell})_{\ell}$ of $H(\A_f)$ and an $\calO$-lattice $L$, put 
$$L\gamma:=\bigcap_{{\ell}<\infty}L\gamma_{\ell}\cap B^2$$
and it is easy to see that $L\gamma$ is also an $\calO$-lattice which is locally equivalent to 
$L$ at each prime $\ell$. 
%Hence we have  
%\begin{equation}\label{desc-general}
%K(L)\bs H(\A_f)/H(\Q)\stackrel{\sim}{\lra} \mL(L)/\sim,\ K(L)\gamma H(\Q)\mapsto  [L\gamma]
%\end{equation}
%where $H(\Q)$ is diagonally embedded in $H(\A_f)$ as $h\mapsto (h)_p$. 

Let us fix an element $L_{{\rm npr}}$ in $\cL_2(1,p)$ so that $K_p(L_{{\rm npr}})$  is the stabilizer of $N_p$ inside $H(\Q_p)$ (see (\ref{Np}) for $N_p$ when $n=2$). By definition of 
$\cL_2(1,p)$, clearly,   
$K_\ell(L_{{\rm npr}})=H(\Z_\ell)$ for any $\ell$ different from $p$.   
Then, we have  
\begin{equation}\label{desc1}
K(L_{{\rm npr}})\bs H(\A_f)/H(\Q)\stackrel{\sim}{\lra} \cL_2(1,p), 
K(L_{{\rm npr}})\gamma H(\Q)\mapsto  [L_{{\rm npr}}\gamma]
\end{equation}
where $H(\Q)$ is diagonally embedded in $H(\A_f)$ as $h\mapsto (h)_p$. 

For each prime $\ell\neq p$, put $K(L_{{\rm npr}})^{(\ell)}=\ds\prod_{q\neq \ell}K_q(L_{{\rm npr}})$. Clearly, $K(L_{{\rm npr}})=K(L_{{\rm npr}})^{(\ell)}\times H(\Z_\ell)$. 
We also define
$$\G_{\mathrm{npr}}^H:= H(\Q)\cap K(L_{{\rm npr}})^{(\ell)}.$$
and denote by $\G_{\mathrm{npr}}^{\GSp_4}$ the corresponding subgroup of $\GSp_4(\Q_\ell)$ under 
(\ref{Ghitza}).

\begin{prop}\label{iko-mb} For each prime $\ell\neq p$, 
there is a one-to-one correspondence between 
$\cL_2(1,p)$ and $$\G_{{\rm npr}}^{\GSp_4}\bs \GSp_4(\Q_\ell)/Z_{GSp_4}(\Q_\ell) \GSp_4(\Z_\ell).$$
\end{prop}
\begin{proof}
Let $H^1=[H,H]$ be the derived group of $H$.  
For any algebraic closed field $F$, $H^1(F)=\Sp_4(F)$. 
Since $Sp_4$ is simply connected as an algebraic group over $\Q$, so is $H^1$. 
Let $\A^{(\ell)}_f$ be the finite adeles of $\Q$ outside $\ell$. 
By the strong approximation theorem (cf.\ Theorem 7.12, p.427 in Section 7.4 of \cite{PR}) for $H^1$ with respect to $S=\{\infty,\ell\}$ and using the exact sequence 
$$1\lra H^1\lra H\stackrel{\nu}{\lra} GL_1\lra 1,$$
with the fact that $\nu(K(L_{{\rm npr}})^{(\ell)})=(\widehat{\Z}^\times)^{(\ell)}$, 
we conclude that $H(\Q)\bs H(\A^{(\ell)}_f)/K(L_{{\rm npr}})^{(\ell)}=\{1\}$. 
It yields a decomposition 
\begin{equation}\label{sapprox1}
H(\A_f)=H(\A^{(\ell)}_f)\times H(\Q_\ell)=H(\Q)(K(L_{{\rm npr}})^{(\ell)}\times H(\Q_\ell)).
\end{equation}

Combining (\ref{desc1}) with (\ref{sapprox1}), we have 
\begin{eqnarray}\label{sapprox2}
\cL(1,p)&\stackrel{\sim}{\lla}&  K(L_{{\rm npr}}) \bs H(\A_f)/H(\Q)\nonumber \\
&\stackrel{g\mapsto \bar{g}^t \atop\sim}{\lra} & H(\Q)\bs H(\A_f)/K(L_{{\rm npr}}) \nonumber \\
&=&H(\Q) \bs H(\Q)(K(L_{{\rm npr}})^{(\ell)}\times H(\Q_\ell))/
K(L_{{\rm npr}})^{(\ell)}\times H(\Z_\ell)\\
&=&(H(\Q)\cap K(L_{{\rm npr}})^{(\ell)}) \bs 
H(\Q_\ell)/ H(\Z_\ell) =\G_{\mathrm{npr}}^H \bs 
H(\Q_\ell)/ H(\Z_\ell) \nonumber \\
&\stackrel{(\ref{Ghitza})}{\simeq} &\G_{{\rm npr}}^{\GSp_4}\bs \GSp_4(\Q_\ell)/
\GSp_4(\Z_\ell) =\G_{{\rm npr}}^{\GSp_4}\bs \GSp_4(\Q_\ell)/Z_{GSp_4}(\Q_\ell) 
\GSp_4(\Z_\ell). \nonumber 
\end{eqnarray}
We remark that at the last line $Z_{GSp_4}(\Q_\ell)$ is intentionally inserted due to the formulation 
in terms of the Bruhat-Tits building of $PGSp_4=GSp_4/Z_{GSp_4}$ handled later on.
\end{proof}

\subsection{A connection to $\ell$-adic symmetric spaces a la Jordan-Zaytman}\label{lJZ}
Let $\ell\neq p$ be a prime.  
As in \cite[Section 2.5]{ATY}, we construct several natural bijections between objects in Theorem \ref{1to1} and 
$$\G_{{\rm npr}}^{\GSp_4}\bs \GSp_4(\Q_\ell)/Z_{GSp_4}(\Q_\ell) 
\GSp_4(\Z_\ell).$$ 
Then, we also check the compatibility of the Hecke operators  
at $\ell$ which are defined on each of objects and on 
the $\ell$-adic symmetric space 
$\GSp_4(\Q_\ell)/Z_{GSp_4}(\Q_\ell) \GSp_4(\Z_\ell)=
\PGSp_4(\Q_\ell)/\PGSp_4(\Z_\ell)$ which can be viewed as 
a special 1-complex of the Bruhat-Tits building of $\PGSp_4(\Q_\ell)$ (\cite[Section 3.3]{ATY}).

As we already stated (see Theorem \ref{1to1}), the set ${\rm MB}(p)$ of Moret-Bailly families up to automorphism is 
in one-to-one correspondence with the set 
\begin{equation}\label{D}
\widetilde{\mathcal D}(A) = \{(A, {\mathcal L}) \mid {\mathcal L}\in {\mathcal D}\}/{\rm Aut}(A).
\end{equation}
Moreover, ${\ell}$-Rechelot isogenies between Moret-Bailly families correspond 
to Rechelot isogenies between $(A, {\mathcal L)}$'s associated with the Moret-Bailly families.
Therefore, in this section, we work with the set $\widetilde{\mathcal D}$ rather than 
$\mathcal{S}_{2,\bF_p}$ or $\MB(p)$. This allows us to investigate the structure of $\Gamma_{{\rm npr}}$ precisely (for the definition of $\Gamma_{{\rm npr}}$, see (\ref{dagger}) below), 
whereas such an analysis does not arise in the case of the principal genus. 

\subsubsection{An interpretation of $\Gamma_{{\rm npr}}$}\label{intGammanpr}
We fix a polarized abelian surface $(A, {\mathcal L})$ with $A = E \times E$ 
and ${\mathcal L} \in {\mathcal D}$,
where $E$ is a supersingular elliptic curve.
Put $g = j({\mathcal L}) \in {\Lambda}$. Then, there exists $x \in \GL_2(B)$ and $m \in \Q_{>0}$
such that $g = m x{\bar{x}}^t$ and ${\mathcal{O}}^2x$ represents an element of the non-principal genus 
${\mathcal{L}}_2(1, p)$. We fix $L_{{\rm npr}} ={\mathcal{O}}^2x$ as a representative of
the non-principal genus.
For the Moret-Bailly family $q:(\mathcal{X},\mM)\lra \bf P^1$ corresponding to
$(A, {\mathcal L})$, we put
\begin{equation}\label{dagger}
\G_{{\rm npr}}({\mathcal X}/\mathbb{P}^1)=\G_{{\rm npr}}(A,\mL):=\{f\in ({\rm End}(A)\otimes_\Z \Z[1/\ell])^\times\ |\ f\circ f^\dagger=
f^\dagger\circ f\in \Z[1/\ell]^\times {\rm id}_A\}.
\end{equation}
Here, ${}^\dagger$ denotes the Rosati involution with respect to ${\mathcal L}$, defined by
$$
        f^{\dagger} = \varphi_{\mathcal L}^{-1}\circ \hat{f} \circ \varphi_{\mathcal L},
$$
where $\hat{f}$ is the dual homomorphism of $f$.

\begin{lem}\label{pull-back}
Let $\ell$ be a prime number distinct from $p$. Then, for $f \in {\rm End}(A)\otimes_{\Z}{\Q}$
and an integer $n$, we have
$f\circ f^{\dagger} = \ell^n {\rm id}_{A}$ if and only if $f^{*}{\mathcal L} = \ell^n {\mathcal L}$.
\end{lem}
\begin{proof}
This lemma follows from the following computation:
$$
  \varphi_{f^{*}{\mathcal L}}
 = \hat{f}\circ \varphi_{\mathcal L}\circ f 
 = \varphi_{\mathcal L}\circ \varphi_{\mathcal L}^{-1}\circ\hat{f}\circ \varphi_{\mathcal L}\circ f
    = \varphi_{\mathcal L}\circ f^{\dagger}\circ f = \varphi_{\mathcal L}\circ \ell^{n}{\rm id}_{A} 
     = \varphi_{\ell^{n} {\mathcal L}}.
$$
\end{proof}
\begin{lem}\label{x}
The homomorphism
$$
\begin{array}{rccc}
     \alpha : &\G_{{\rm npr}}(A,\mL) & \longrightarrow & \Gamma_{{\rm npr}}^H\\
             &  r &\mapsto &  x^{-1} {\bar{r}}^tx
\end{array}
$$
is an isomorphism.
\end{lem}
\begin{proof}
First, we prove that the homomorphism $\alpha$ is well defined.
Let $r \in \G_{{\rm npr}}(A,\mL) $ and set $y = x^{-1} {\bar{r}}^tx$.
By Lemma \ref{pull-back},
we have $r^{*}{\mathcal L} = \ell^n{\mathcal L}$. Hence, by Theorem \ref{intersection} 
(with the intersection formula therein)
$$
         {\bar{r}}^t g r = \ell^n g,\ g=j(\mL).
$$
Since $g = m x{\bar{x}}^t$ with $m\in \Q^\times_{>0}$, it follows that
$$
       (x^{-1}{\bar{r}}^tx)\overline{(x^{-1}{\bar{r}}^tx)}^{t} = \ell^n{\rm id}_A.
$$
Therefore, $y = x^{-1} {\bar{r}}^tx \in G_{2}(\Q)= H(\Q)$. 
Now suppose that $q \neq \ell$. Since ${\bar{r}}^t \in G_{2}({\Z}_q)= H({\Z}_q)$,
we obtain
$$
  (L_{{\rm npr}}\otimes_{\Z}{\Z}_q)y = ({\mathcal O}^2\otimes_{\Z}{\Z}_q )xy 
  = ({\mathcal O}^2\otimes_{\Z}{\Z}_q){\bar{r}}^tx = ({\mathcal O}^2\otimes_{\Z}{\Z}_q)x 
  = (L_{{\rm npr}}\otimes_{\Z}{\Z}_q).
$$
Hence, $y \in \Gamma_{{\rm npr}}^H$.

It is straightforward to see that $\alpha$ is injective.

To show the surjectivity, let $y \in \Gamma_{{\rm npr}}^H$. Since $y \in H(\Q) = G_{2}(\Q)$,
we have $y \bar{y}^t= \bar{y}^t y = \nu (y) {\rm id}_A$. Set $r = \overline{xyx^{-1}}^t$.
Then $r \in \GL_2(B)$. For each $q \neq \ell$, we have
$$ 
({\mathcal O}^2\otimes_{\Z}{\Z}_q )x = L_{{\rm npr}}\otimes_{\Z}{\Z}_q = (L_{{\rm npr}}\otimes_{\Z}{\Z}_q)y= ({\mathcal O}^2\otimes_{\Z}{\Z}_q )xy
=({\mathcal O}^2\otimes_{\Z}{\Z}_q ){\bar{r}}^tx.
$$
Hence, $r \in \GL_2({\mathcal O}_q)$  ($q \neq \ell$). Therefore, $r \in \GL_2(\mathcal O [1/\ell])$.
Substituting $y = x^{-1}\bar{r}^t x$ into $y \bar{y}^t= \nu (y) {\rm id}_A$, we obtain
$\nu(y) x \bar{x}^t = \bar{r}^t x \bar{x}^t r$. Since $g = mx \bar{x}^t$,
it follows that $\nu(y) g= \bar{r}^t g r$.
Hence, by Theorem \ref{intersection}, we have $r^*{\mathcal L} = \nu(y){\mathcal L}$.
Furthemore,
$$\deg(r) ({\mathcal L}\cdot {\mathcal L})= (r^*{\mathcal L}\cdot r^*{\mathcal L}) =
\nu(y)^2({\mathcal L}\cdot {\mathcal L}),
$$
and therefore $\deg(r) = \nu (y)^2$.
Since $r \in \GL_2(\mathcal O [1/\ell])$, there exists an integer $n$ such that $\deg(r) = \ell^n$.
Thus $n$ is even and $\nu (y)^{n/2}\in \Z[1/\ell]^\times$. Hence, 
$r\cdot r^{\dagger} = r^{\dagger}\cdot r = \nu (y)^{n/2}{\rm id}_A$.
Hence $r \in \G_{{\rm npr}}(A,\mL)$, which proves that 
the homomorphism $\alpha$ is surjective.
\end{proof}

We denote by $T_\ell(A)$ the $\ell$-adic Tate module of $A$, and 
by $V_\ell(A):=T_\ell(A)\otimes_{\Z_\ell}\Q_\ell$ the rational $\ell$-adic Tate module of $A$ 
(cf. Section 18 of Chapter IV of \cite{M}).  
Fix a standard symplectic basis of $V_\ell(A)$, which yields an isomorphism 
\begin{equation}\label{standard}
 {\rm Aut}(V_\ell(A),\langle \ast,\ast \rangle_{\mL})\simeq \GSp_4(\Q_\ell).
\end{equation}
Let $$({\rm End}(A,\mL)\otimes_\Z\Q_\ell)^\times:=\{g\in ({\rm End}(A)\otimes_\Z\Q_\ell)^\times\ |\ g\cdot g^\dagger\in 
\Q_\ell{\rm id}_A\}$$ where $g^\dagger$ is the Rosati involution with respect to $\mL$.  
By functoriality, we have the natural homomorphism
\begin{equation}\label{representation}
\begin{array}{cccc}
({\rm End}(A,\mL)\otimes_\Z\Q_\ell)^\times & \hookrightarrow & 
{\rm Aut}(V_\ell(A),\langle \ast,\ast \rangle_{\mL})  & \stackrel{(\ref{standard})}{\simeq} \GSp_4(\Q_\ell)\stackrel{(\ref{Ghitza})}{\simeq} H(\Q_\ell).\\
 r  &\mapsto & r^\ast &   
\end{array}
\end{equation}
By the Tate conjecture (cf. \cite{Mori},\cite{Zarhin}), the above first inclusion map is surjective and thus 
the morphism (\ref{representation}) is an isomorphism. 

Recall that $x\in \GL_2(B)$ satisfies $g=m x \bar{x}^t$ with $m\in \Q^\times_{>0}$. 
By strong approximation, we can choose $\Pi\in B$ such that $\Pi\cdot \overline{\Pi}=\ell$. 
Replacing $x$ with $x\Pi^n,\ n\in \Z$ if necessary, we may assume $\ord_\ell(m)=0$. 
Since $g\in \Lambda$ and $p\neq \ell$, $g\in \GL_2(\calO_\ell)$ and thus $x\in \GL_2(\calO_\ell)$. 
Applying the argument for the proof of Lemma \ref{x} to the identification 
$H(\Q_\ell)
\stackrel{(\ref{representation})}{\simeq} ({\rm End}(A,\mL)\otimes_\Z\Q_\ell)^\times$, 
the isomorphism $$\beta:H(\Q_\ell)\lra H(\Q_\ell), \gamma\mapsto  x^{-1}\bar{r}^t x
$$
is well-defined and is easy to check that $\beta(H(\Z_\ell))=H(\Z_\ell)$ and $\nu(x^{-1}\bar{r}^t x)=\nu(\gamma)$. 

\begin{cor}\label{iden-npr} Under the isomorphism 
$$
({\rm End}(A,\mL)\otimes_\Z\Q_\ell)^\times \stackrel{(\ref{representation})\atop \sim}{\lra} 
H(\Q_\ell) \stackrel{\beta\atop \sim}{\lra} H(\Q_\ell), 
$$
the subgroup $\G_{{\rm npr}}({\mathcal X}/\mathbb{P}^1)=\G_{{\rm npr}}(A,\mL)$ of $({\rm End}(A,\mL)\otimes_\Z\Q_\ell)^\times $ 
is identified with the subgroup $\G_{{\rm npr}}^H \subset H(\Q_\ell)$ in Proposition \ref{iko-mb}. 
\end{cor}

\subsubsection{$\ell$-markings}\label{ellmarkings}
In view of the compatibility of 
 Hecke operators on the various stages, to relate 
 $ \widetilde{\mathcal D}(A)$ 
 ($\simeq \mathcal{D}/{\rm Aut}(A)\simeq \MB(p)$) with  $\G_{{\rm npr}}^H\bs H(\Q_\ell)/H(\Z_\ell)$ or 
 $\G_{{\rm npr}}^{\GSp_4}\bs \GSp_4(\Q_\ell)/Z_{GSp_4}(\Q_\ell) \GSp_4(\Z_\ell)$, 
 as in \cite[Section 2.5]{ATY}, we use the notation of 
 $\ell$-markings. 
 
 Fix $(A,\mL_0)\in \widetilde{\mathcal{D}}(A)$. An isogeny $\phi:(A,\mL_0)\lra (A,\mL)$ is called an $\ell$-marking from 
 $(A,\mL_0)$  
 if $\phi^\ast \mL=\mL_0^{\ell^n}$ for some $n\in \Z_{\ge 0}$. 
 %A quasi-isogeny $\phi:(A,\mL_0)\lra (A,\mL)$ is called an $\ell$-marking if 
 %$\ell^m \phi$ is an $\ell$-marking for some $m\in \Z_{\ge 0}$. Henceforth, $\ell$-markings 
 %means quasi-isogenies.   
 Two $\ell$-markings $\phi_i:(A,\mL_0)\lra (A,\mL_i),\ i=1,2$ are said to be isomorphism if there exists an isomorphism 
 $h:(A,\mL_1)\lra (A,\mL_2)$ such that $h\circ \phi_1=\phi_2$. Consider two $\ell$-markings
 $\phi_i:(A,\mL_0)\lra (A,\mL),\ \phi^\ast_i \mL=\mL_0^{\ell^{n_i}},\ n_i\in\Z,\ i=1,2$ 
 with the same target. Then, $\sqrt{\deg(\phi_i)}=\ell^{n_i}$. 
It is easy to see that 
$$f:=\frac{1}{\ell^{n_1}}\phi^{\dagger\dagger}_1 \circ \phi_2\in {\rm End}(A)\otimes_
\Z \Z[\frac{1}{\ell}],\ \phi^{\dagger\dagger}_1:=\vp_{\mL_0}^{-1}\circ \widehat{\phi}_1\circ
\vp_{\mL}.$$
Further, put $f^\dagger:=\vp_{\mL_0}^{-1}\circ \widehat{f} \circ \vp_{\mL_0}$ and then 
\begin{eqnarray*}
f^\dagger\circ f&=&\vp_{\mL_0}^{-1}\circ \widehat{f} \circ \vp_{\mL_0}\circ f \\
&=&\ell^{-2n_1} \vp_{\mL_0}^{-1}\circ \widehat{\phi}_2\circ 
(\widehat{\phi^{\dagger\dagger}_1}  \circ \vp_{\mL_0} \circ \phi^{\dagger\dagger}_1) \circ \phi_2 \\
&=&\ell^{-2n_1} \vp_{\mL_0}^{-1}\circ \widehat{\phi}_2\circ 
(\sqrt{\deg(\phi^{\dagger\dagger}_1}) \vp_{\mL})\circ \phi_2 \\
&=&\ell^{-2n_1} \sqrt{\deg(\phi^{\dagger\dagger}_1)} \vp_{\mL_0}^{-1}\circ (\widehat{\phi}_2\circ 
\vp_{\mL}\circ \phi_2) \\
&=&\ell^{-2n_1} \sqrt{\deg(\phi^{\dagger\dagger}_1)} \vp_{\mL_0}^{-1}\circ (
\sqrt{\deg(\phi_2)}\vp_{\mL_0}) \\
&=&\ell^{-2n_1} \sqrt{\deg(\phi_1)\deg(\phi_2)} {\rm id}=\ell^{n_2-n_1} {\rm id}
\end{eqnarray*}
and similarly, $f\circ f^\dagger=\ell^{n_2-n_1}{\rm id}$. Thus, two $\ell$-markings differ by 
an element of $\G_{{\rm npr}}(A,\mL_0)$ by Lemma \ref{pull-back}. 

Let ${\rm M}((A,\mL_0),\ell)$ be the set of all isomorphism classes of $\ell$-markings from $(A,\mL_0)$. 
The group $\G_{{\rm npr}}(A,\mL_0)$ acts on  ${\rm M}((A,\mL_0),\ell)$ via composition. 
Put 
$$
{\rm Iso}((A,\mL_0),\ell):={\rm M}((A,\mL_0),\ell)/\G_{{\rm npr}}(A,\mL_0).
$$
Then, obviously,  
$$
{\rm Iso}((A,\mL_0),\ell)=\{[(A,\mL)]\in \widetilde{\mathcal D}(A)\ |\ 
{}^\exists\phi:(A,\mL_0)\lra (A,\mL) 
\text{a quasi-isogeny with $\ell$-power degree} \}.
$$
By definition, we see easily that ${\rm Iso}((A,\mL_0),\ell)\subset \widetilde{\mathcal{D}}(A)$. 

\subsubsection{A correspondence a la Jordan-Zaytman}
In this subsection, we adopt the argument of Jordan-Zaytman (see around \cite[Theorem 47]{JZ}) to our case. 
Let us keep the notation in previous sections so that we have fixed $(A,\mL_0)\in \widetilde{\mathcal D}(A)$. 
Set $A[\ell^\infty]= \cup_{n \geq 1} A[{\ell}^n]$ and consider the exact sequence 
$$
0\lra T_\ell(A)\lra V_\ell(A)\stackrel{\pi}{\lra} V_\ell(A)/T_\ell(A)\simeq A[\ell^\infty]\lra 0.
$$
For each $\ell$-marking $\phi:(A,\mL_0)\lra (A,\mL)$, put $C_\phi:={\rm Ker}(\phi)\subset 
A[\ell^\infty]$ and consider the lattice 
$$
T(C_\phi):=\pi^{-1}(C_\phi)\subset V_\ell(A).
$$
We denote by $P((A,\mL),\phi)\in \GSp_4(\Q_\ell)$ the matrix representation of $T(C_\phi)$ 
with respect to the basis regarding ${\rm Aut}(V_\ell(A),\langle \cdot,\cdot \rangle_{\mL_0})\stackrel{(\ref{standard})}{\simeq} \GSp_4(\Q_\ell)$.  
For each, $f\in \G(A,\mL_0)$, we see that  
$$T(C_{\phi\circ f})=f^\ast (T(C_\phi)).$$
Thus, $f^\ast$ acts on $P((A,\mL),\phi)$ from the left. 
Then, we have a well-defined map  
\begin{equation}\label{iso}
{\rm Iso}((A,\mL_0),\ell)\lra \G_{{\rm npr}}(A,\mL_0)\bs \GSp_4(\Q_\ell)/\GSp_4(\Z_\ell),\ (A,\mL),\phi\mapsto P((A,\mL),\phi).
\end{equation}
We identify $({\rm End}(A,\mL)\otimes_\Z \Q_\ell)^\times$ with $\GSp_4(\Q_\ell)$. 
The subgroup $({\rm End}(A,\mL)\otimes_\Z \Z_\ell)^\times:=
{\rm End}(A)\otimes_\Z \Z_\ell \cap ({\rm End}(A,\mL)\otimes_\Z \Q_\ell)^\times$ 
is identified with  $\GSp_4(\Z_\ell)$. 
Then, by Corollary \ref{iden-npr} and (\ref{Ghitza}), we have 
\begin{equation}\label{iso2}
\G_{{\rm npr}}(A,\mL_0)\bs \GSp_4(\Q_\ell)/\GSp_4(\Z_\ell)\stackrel{\text{Corollary 3.4}}{\simeq}  \G_{{\rm npr}}^H\bs H(\Q_\ell)/H(\Z_\ell)
\stackrel{(\ref{Ghitza})}{\simeq}  \G_{{\rm npr}}^{\GSp_4}\bs \GSp_4(\Q_\ell)/\GSp_4(\Z_\ell)  
\end{equation}

Summing up everything, we have proved the following result:
\begin{prop}\label{JZ}There is a natural bijection
\begin{eqnarray*}
\mathcal{S}_{2,\bF_p}\simeq \MB(p)\simeq \widetilde{\mathcal D}(A) &\stackrel{{\rm inclusion}\atop \sim}{\longleftarrow}& 
{\rm Iso}((A,\mL_0),\ell) \\ 
&\stackrel{(\ref{iso})\atop \sim}{\lra}& \G_{{\rm npr}}(A,\mL_0)\bs \PGSp_4(\Q_\ell)/\PGSp_4(\Z_\ell)\\
&\stackrel{(\ref{iso2})}{\simeq}&  \G_{{\rm npr}}^H\bs H(\Q_\ell)/H(\Z_\ell) \\
&\stackrel{(\ref{Ghitza})}{\simeq}& \G_{{\rm npr}}^{\GSp_4}\bs \GSp_4(\Q_\ell)/\GSp_4(\Z_\ell)  
\end{eqnarray*}
\end{prop}
\begin{proof} The correspondence (\ref{iso}) is surjective by the argument for \cite[p.33, Theorem 47]{JZ}. 
Then, we have 
$$|\G_{{\rm npr}}^H\bs H(\Q_\ell)/H(\Z_\ell)|=
|\G_{{\rm npr}}(A,\mL_0)\bs \PGSp_4(\Q_\ell)/\PGSp_4(\Z_\ell)|\le |{\rm Iso}((A,\mL_0),\ell) |\le |\widetilde{\mathcal D}(A)|=|\mathcal{D}|.$$
The both sides have the same cardinality by Theorem \ref{1to1} and Proposition \ref{iko-mb}. 
Thus, we have the claim. 
\end{proof}

\subsection{The Hecke operator at $\ell$}\label{hecke-at-ell}
Finally we discuss compatibility of the maps in Proposition \ref{JZ} with the 
Hecke operator at $\ell$ defined on each finite set. 
We refer Section 3 in Chapter VII of \cite{CF} for a usual setting and Section 16 through 19 of 
\cite{vdGeer} 
%as a reader's friendly reference. 
as a reference for the reader's convenience.
The Hecke operators are similarly defined also in our setting.

We recall the construction of the relative $(\ell,\ell)$-isogenies for  
Moret-Baily families  in Section \ref{reliso}. 

For each prime $\ell$ different from $p$ and a class $[(\calX,\mM)]$ in 
a class of $\mathcal{S}_{2,\bF_p}$, we define 
the (geometric) Hecke correspondences $T(\ell)^{{\rm geo}}$ at $\ell$: 
\begin{equation}\label{geo-Hecke}
T(\ell)^{{\rm geo}}([(\calX,\mM)]):=\sum_{G\subset K_2(\mL)\atop 
\text{maximal isotropic}}[(\calX',\mM')]. 
\end{equation}
where $(\calX',\mM')$ is the Moret-Baily family associated to $(A/G \cong A,\mL')$ 
defined in Section \ref{reliso}. 
For $\MB(p)$ and $\widetilde{\mathcal D}$, the (geometric) Hecke correspondences $T(\ell)^{{\rm geo}}$ at $\ell$ is defined in the same way on each of $\MB(p),\ \widetilde{\mathcal D}(A)$ 
and they are compatible under the identification $\mathcal{S}_{2,\bF_p}\simeq \MB(p)\simeq \widetilde{\mathcal D}(A)$.

The geometric Hecke operator at $\ell$ on ${\rm Iso}((A,\mL_0),\ell)={\rm M}((A,\mL_0),\ell)/\G_{{\rm npr}}(A,\mL_0)$ is defined by 
letting $T(\ell)^{{\rm geo}}$ acts on the target of a representative $\phi:(A,\mL_0)\lra (A,\mL)$ of 
each class of  ${\rm Iso}((A,\mL_0),\ell)$. Namely, 
\begin{equation}\label{geo-Hecke2}
T(\ell)^{{\rm geo}}([(A,\mL_0)\stackrel{\phi}{\lra} (A,\mL)]):=\sum_{G\subset K_2(\mL)\atop 
\text{maximal isotropic}}[(A,\mL_0)\stackrel{\phi}{\lra} (A,\mL)\lra (A,\mL')]  
\end{equation}
where $(A,\mL)\lra (A,\mL')$ is the $(\ell,\ell)$-isogeny corresponding to $G\subset K_2(\mL)$ in the setting of 
$\widetilde{\mathcal D}(A)$. 
This operator is compatible with $T(\ell)^{{\rm geo}}$ on $\widetilde{\mathcal D}(A)$ under 
the identification in Proposition \ref{JZ}. 

Thus, we henceforth identify the above four geometric Hecke correspondences under 
the identification in Proposition \ref{JZ}. 

Recall $H(\Q_\ell)\stackrel{(\ref{Ghitza})}{\simeq} \GSp_4(\Q_\ell)=\GSp(\Q^{4}_\ell,\ \langle \ast,\ast \rangle)$ where 
$\langle \ast,\ast \rangle$ is the standard symplectic pairing on $\Q^{4}_\ell\times
\Q^{4}_\ell$. Put $V=\Q^{4}_\ell$. 
Each element of $\GSp_4(\Q_\ell)/\GSp_4(\Z_\ell)$ can be regarded as a lattice 
$L$ of $V$ such that $\langle \ast,\ast \rangle_{L\times L}$ gives a $\Z_\ell$-integral  
symplectic structure on $L$ (cf. \cite[Section 4]{ATY}). 
Using this interpretation, 
 each element of $\GSp_4(\Q_\ell)/Z_{GSp_4}(\Q_\ell) \GSp_4(\Z_\ell)=
 \PGSp_4(\Q_\ell)/\PGSp_4(\Z_\ell)$ can be regarded as 
 a homothety class $[L]$ for such an $L$. 
For each $L$ being as above, we define the Hecke correspondence on 
$\GSp_4(\Q_\ell)/\GSp_4(\Z_\ell)$ at $\ell$
\begin{equation}\label{Hecke-ope}
T(\ell)^{{\rm BT}}([L]):=\sum_{L\subset L_1 \subset \ell^{-1}L  \atop  
L_1/L\text{:maximal isotropic}}[L_1]
\end{equation}
where $L_1$ runs over all lattices enjoying 
$L\subset L_1 \subset \ell^{-1}L$ as denoted and that 
$L_1/L$ is a maximal isotropic subgroup of $\ell^{-1}L/L$ 
with respect to the symplectic pairing $\langle \ast,\ast \rangle_{\ell^{-1}L/L\times\ell^{-1}L/L}$. 
Clearly, the action of $H(\Z[1/\ell])$ (given by multiplication from the left) on lattices are equivariant under 
$T(\ell)$ and thus the same is true for $\G_{{\rm npr}}^H\subset H(\Z[1/\ell])$. Therefore, it also induces a correspondence on 
\begin{equation}\label{iso3}
\G_{{\rm npr}}^H\bs H(\Q_\ell)/H(\Z_\ell)=
\G_{{\rm npr}}^H\bs H(\Q_\ell)/Z_H(\Q_\ell)H(\Z_\ell)
\stackrel{(\ref{iso2})}{\simeq} \G_{{\rm npr}}(A,\mL_0)\bs \GSp_4(\Q_\ell)/Z_{GSp_4}(\Q_\ell) \GSp_4(\Z_\ell)
\end{equation}
and we also denote it by $T(\ell)^{{\rm BT}}$ on each of both sides. 
We should remark that under the above identification, 
the Hecke operator associated to $\GSp_4(\Z_\ell)t_\ell \GSp_4(\Z_\ell)
,\ t_\ell=\diag(1,1,\ell,\ell)$ on the left hand side of (\ref{iso3})  
corresponds to one associated to $H(\Z_\ell)x^{-1}\diag(1,\ell)x H(\Z_\ell)$ 
(note that $\beta(H(\Z_\ell))=H(\Z_\ell)$ as mentioned before). Since 
$\nu(x^{-1}\diag(1,\ell)x)=\nu(\diag(1,\ell))=\ell$. Thus, 
$H(\Z_\ell)x^{-1}\diag(1,\ell)x H(\Z_\ell)=H(\Z_\ell)\diag(1,\ell)H(\Z_\ell)$. 
This means that under the identification (\ref{iso2}), the standard Hecke operator at $\ell$ 
is preserved.

For a set $S$, we define ${\rm Div}(S)_\Z:=\bigoplus_{P\in S}\Z P$ and call it the 
divisor class group of $S$. 
Using this notaion,  the Hecke operators are extended uniquely to a 
additive homomorphism on the divisor class group in question.   
We have obtained the following result which follows from the 
definition of Hecke operators and Proposition \ref{JZ}:
\begin{thm}\label{Rel-HP}
Let $S\in \{\mathcal{S}_{2,\bF_p}, \MB(p), \widetilde{\mathcal D}(A)\}$ and put $S_2:=\GSp_4(\Q_\ell)/Z_{GSp_4}(\Q_\ell) \GSp_4(\Z_\ell)$.  
The following diagram is commutative:
{\small
\[
\xymatrix@C=1em@R=2.5em{
{\rm Div}(S)_\Z
\ar[d]_{T(\ell)^{{\rm geo}}}
&
{\rm Div}({\rm Iso}((A,\mL_0),\ell))_\Z
\ar[l]_{\sim\hspace{8mm}}
\ar[r]^{\sim\hspace{2mm}}
\ar[d]_{T(\ell)^{{\rm geo}}}
&
{\rm Div}(\G_{{\rm npr}}(A,\mL_0)\bs S_2)_\Z
\ar[r]^{\sim\hspace{4mm}}
\ar[d]_{T(\ell)^{{\rm BT}}}
&
{\rm Div}(\G_{{\rm npr}}^H\bs H(\Q_\ell)/H(\Z_\ell))_\Z
\ar[d]_{T(\ell)^{{\rm BT}}}
\ar[r]^{\hspace{18mm}\sim}&
\\
{\rm Div}(S)_\Z
&
{\rm Div}({\rm Iso}((A,\mL_0),\ell))_\Z
\ar[l]_{\sim\hspace{8mm}}
\ar[r]^{\sim\hspace{2mm}}
&
{\rm Div}(\G_{{\rm npr}}(A,\mL_0)\bs S_2)_\Z
\ar[r]^{\sim\hspace{4mm}}
&
{\rm Div}(\G_{{\rm npr}}^H\bs H(\Q_\ell)/H(\Z_\ell))_\Z 
\ar[r]^{\hspace{18mm}\sim} &
}
\]
}
{\small
\[
\xymatrix@C=1em@R=2.5em{
{\rm Div}(\G_{{\rm npr}}^{\GSp_4}\bs \GSp_4(\Q_\ell)/\GSp_4(\Z_\ell))_\Z
\ar[d]_{T(\ell)^{{\rm BT}}}
\\
{\rm Div}(\G_{{\rm npr}}^{\GSp_4}\bs \GSp_4(\Q_\ell)/\GSp_4(\Z_\ell))_\Z
}
\]
}

where the horizontal arrows are induced from Proposition \ref{JZ}. 
\end{thm}

\subsection{Row vectors and column vectors}\label{column}
In the next section, we  use the theory of Bruhat-Tits buildings.
In this theory, we use ${\mathcal O}$-lattices as lattices of column vectors.
In this subsection, we explain how to pass from the row vector expression
to the column vector expression. The notion of ${\mathcal O}$-lattices as lattices of row vectors
is naturally translated into that of ${\mathcal O}$-lattices of column vectors.
Let ${\mathcal L}'(1, p)$ denote the set of maximal ${\mathcal O}$-lattices consisting 
of column vectors.
Each element of ${\mathcal L}'(1, p)$ is a lattice which is also a right ${\mathcal O}$-module.
For an ${\mathcal O}$-lattice $L$, we set
$$
        \bar{L}^t = \Big\{
\left(
\begin{array}{c}
\bar{a} \\
\bar{b}
\end{array}
\right)
\ \Big|\ 
(a, b) \in L\Big\}.
$$
We consider the following map:
$$
\begin{array}{rccccc}
   T: &  {\mathcal L}(1, p) & \longrightarrow & {\mathcal L}'(1, p)  &    &     \\
    &   L   & \mapsto &   \bar{L}^t.
\end{array}
$$
This map $T$ is bijective and compatible with the global equivalence.
As in the case of the row vector expression, for an element $L'$ of ${\mathcal L}'(1, p)$ 
there exists $x' \in GL_2(B)$ such that $L' = x'{\mathcal O}^2$, 
where ${\mathcal O}^2$ is written as column vectors.
If $L = {\mathcal O}^2x$ with $x \in GL_2(B)$, then
$$
\bar{L}^t =\Big\{{\bar x}^t
\left(
\begin{array}{c}
\bar{a} \\
\bar{b}
\end{array}
\right)\ 
\Big|  \ 
a, b\in {\mathcal O}\Big\}
= \Big\{{\bar x}^t
\left(
\begin{array}{c}
a \\
b
\end{array}
\right)\ 
\Big|\   
a, b\in {\mathcal O}\Big\} = \bar{x}^t{\mathcal O}^2.
$$
Therefore, even if we adopt the convention that lattices are represented 
by column vectors, the situation does not change under the correspondence given by $T$ .

\subsection{The graph defined by the special 1-complex}\label{btq1} 
Put $G=\GSp_4(\Q_\ell),\ K=\GSp_4(\Z_\ell)$, and $Z=Z_{GSp_4}(\Q_\ell)$ 
for simplicity. 
In this subsection, 
we consider the graph associated to the quotient $\G_{{\rm npr}}\bs S_2$ 
where $S_2=G/ZK= \PGSp_4(\Q_\ell)/\PGSp_4(\Z_\ell)$. 
The symmetric space $S_2$ is called 
the special 1-complex of the Bruhat-Tits building of $PGSp_4$.
In our treatment of the Bruhat-Tits building, we adopt the convention that lattices are represented 
by column vectors and that the group acts on them
by left multiplication. As explained in Subsection \ref{column}, 
this convention does not affect the situation.

The edges are defined by $T(\ell)^{{\rm BT}}$ and thus in terms of lattices (see (\ref{Hecke-ope})). 
Let $t_\ell:={\rm diag}(1,1,\ell,\ell)\in G$.   
We decompose 
\begin{equation}\label{double-coset}
Kt_\ell K=\coprod_{t\in T}g_t K
\end{equation}
where $T$ is the index set so that $|T|=(\ell^2+1)(\ell+1)$. 
Note that for $k \in K$, the coset $kt_{\ell}K$ corresponds to the lattice 
$L = kt_{\ell}({\mathbf Z}_{\ell}^{\oplus 4})$, satisfying 
$\ell ({\mathbf Z}_{\ell}^{\oplus 4})\subset L \subset ({\mathbf Z}_{\ell}^{\oplus 4})$.
We can rewrite (\ref{Hecke-ope}) as 
\begin{equation}\label{Hecke-ope2}
T(\ell)^{{\rm BT}}([L])=\sum_{t\in T}[g_t L ].
\end{equation}
In this vein, two elements $v_1=\G_{{\rm npr}} g_1 ZK$ and 
$v_2=\G_{{\rm npr}} g_2 ZK$ in $\G_{{\rm npr}}\bs G/ZK$ said to be 
adjacent if $v_2=\G_{{\rm npr}} g_1 g_t ZK$ for some $t \in T$ where 
$\{g_t\}_{t\in T}$ is defined in (\ref{double-coset}). In this setting, 
we have a directed edge $v_1\lra v_2$. 

For each element $v=\G_{{\rm npr}} g ZK$, we define 
$${\rm RA}(v):=(\G_{{\rm npr}}\cap g ZK g^{-1})Z/Z$$
which naturally acts on the set $\{\G_{{\rm npr}}g g_tZK\ |\ t\in T\}$. 
For each edge $f:v_1\lra v_2$, we define the weight of $f$ by 
$$w(f):={\rm RA}(v_1)/{\rm Stab}_{{\rm RA}(v_1)}(v_2).$$

The graph in question, say ${\rm BTQ}^1_2(\ell,p)$, is a directed (regular) graph where
\begin{itemize}
\item the set of vertices $V({\rm BTQ}^1_2(\ell,p))$ is $\G_{{\rm npr}}\bs G/ZK$, and
\item the set of directed edges between two vertices $v_1=\G_{{\rm npr}} g_1 ZK$ and 
$v_2=\G_{{\rm npr}} g_2 ZK$ is defined by the adjacency condition in the above sense. 
Namely, an edge from $v_1$ from $v_2$ is $g_t$ with $t\in T$ such that 
$v_2=\G_{{\rm npr}} g_1 g_t ZK$. 
Two edges from $v_1$ to $v_2=\G_{{\rm npr}} g_1 g_t ZK$ and $v'_2=\G_{{\rm npr}} g_1 g'_t ZK$ 
($g_t,g'_t\in T$) are identified if there exists $h=z g_1 k g^{-1}_1\in {\rm RA}(v_1)$ 
such that $k g'_2 ZK=g_2 ZK$. 
\end{itemize} 

\subsection{Comparison theorem}

We will prove the following comparison theorem which plays an important role in 
our study:
\begin{thm}\label{comparison}
The identification of Proposition \ref{JZ} induces the 
following graph isomorphism
$${\mathcal G}^{{\rm MB}}(\ell, p)
\stackrel{\sim}{\lra}{\rm BTQ}^1_2(\ell,p)$$
whose Random walk matrices are given by 
$\frac{1}{(\ell^2+1)(\ell+1)}T(\ell)^{{\rm geo}}$ and 
$\frac{1}{(\ell^2+1)(\ell+1)}T(\ell)^{{\rm BT}}$ on each side respectively.  
Further, the following properties are preserved under the isomorphisms:
\begin{itemize}
\item the Hecke action of $T(\ell)^{{\rm geo}}$ or 
$T(\ell)^{{\rm BT}}$ on each set of the vertices defines neighbors of a given vertex $v$ such that  
the weighted sum of the neighbors satisfies $\ds\sum_{v'\in V,\ f:v\to v'}w(f)=(\ell^2+1)(\ell+1)$,  
\item each edge $f$ from $v_1$ to $v_2$ has an opposite $\widehat{f}$ such that  
$$|{\rm RA}(v_2)|\cdot |O_{{\rm RA}(v_1)}({\rm Ker}(f))|=|{\rm RA}(v_1)|\cdot 
|O_{{\rm RA}(v_2)}({\rm Ker}(\widehat{f}))|$$
where ${\rm Ker}(f)$ is $G$ if $v_1=[(\calX,\mL)]$ and $f$ is defined 
by a maximal totally isotropic subgroup $G\subset K_2(\mL)$ in the case of ${\mathcal G}^{{\rm MB}}(\ell, p)$, 
and $v_2 = \G_{{\rm npr}}g_1 g_t ZK$ 
if $v_1=\G_{{\rm npr}}g_1 ZK$ gives rise to an edge $v_1\lra v_2$
defined by $g_t,\ t\in T$ 
in the case of ${\rm BTQ}^1_g(\ell,p)$.
\end{itemize}
\end{thm}
\begin{proof}
The compatibility of the Hecke operators, hence adjacency matrices follows from Theorem \ref{Rel-HP} 
and this yields the first property in the claim. The remaining formula follows 
from the argument for \cite[Proposition 2.12]{ATY}. 
\end{proof}
We define the finite Hilbert space $\ell^2({\rm BTQ}^1_g(\ell,p))$
associated with ${\rm BTQ}^1_g(\ell,p)$ in a manner similar to that for
${\mathcal G}^{{\rm MB}}(\ell,p)$ introduced in Section~\ref{intro}.

\begin{cor}\label{comp1}The identification of Proposition \ref{JZ} induces the 
following isomorphism 
$$\ell^2({\mathcal G}^{{\rm MB}}(\ell, p))\stackrel{\sim}{\lra}\ell^2({\rm BTQ}^1_g(\ell,p))$$
 as a finite dimensional Hilbert space so that it preserves the inner products and 
 two randam walk operators corresponds each other.
\end{cor}

\begin{cor}\label{lwm}Keep the notation being as above. 
The graph ${\mathcal G}^{{\rm MB}}(\ell, p)$ is connected. 
\end{cor}
\begin{proof} Since the special 1-complex $S_2=\PGSp_4(\Q_\ell)/\PGSp_4(\Z_\ell)$ 
is connected as a subcomplex of the Bruhat-Tits building of $PGSp_4$ by \cite[Proposition 4.3]{ATY}, 
its quotient ${\rm BTQ}^1_2(\ell,p)$ is also connected. 
The claim follows from Theorem \ref{comparison}. 
\end{proof}

\section{Algebraic modular forms}\label{AMFs}
In this section, we refer to \cite[Section 6]{ATY} or to \cite[Chapter II]{Gross}.  

Let $\A_\Q$ be the ring of adeles of $\Q$. 
Put $U_{{\rm npr}}:=K(L_{{\rm npr}})$ for simplicity.
Since $B$ is definite, $H(\R)=G_2(\R)={\rm GUSp}_4(\R)$ is compact modulo its center. It follows from  (\ref{desc1}) that 
\begin{equation}\label{strong-approx} 
H(\Q)\bs H(\A_\Q)/(U_{{\rm npr}}\times H(\R)^+)=H(\Q)\bs H(\A_f)/U_{{\rm npr}}
\end{equation}
where $H(\R)^+$ stands for the connected component of the identity element. 
According to Chapter II-4 of \cite{Gross}, we define 
the $\C$-vector space $M(U_{{\rm npr}})$ consisting of all locally constant functions $f:H(\A_\Q)\lra \C$ such that 
$$f(\gamma g k g_\infty)=f(g),\ g\in H(\A_\Q)$$
for all $\gamma\in H(\Q),\ k\in U_{{\rm npr}}$, and $g_\infty\in H(\R)^+$. 
Put $h:=H_2(1,p)$ for simplicity and pick $\{\gamma_i\}_{i=1}^{h}$ with $\gamma_i\in H(\A_f)$ a complete system of the 
representatives of (\ref{strong-approx}). 
By definition, the space $M(U_{{\rm npr}})$ is generated by 
the characteristic functions $\varphi_i,\ 1\le i\le h$ of 
$H(\Q)\gamma_i U_{{\rm npr}}$. Hence we have 
$M(U_{{\rm npr}})\simeq \C^{\oplus h}$. 
We define a hermitian inner product $( \ast,\ast )$ on $M(U_{{\rm npr}})$ by 
\begin{equation}\label{pairing-alg}( f_1,f_2 ):=\sum_{\gamma\in H(\Q)\bs H(\A_f)/U_{{\rm npr}}}
f_1(\gamma)\overline{f_2(\gamma)}\frac{1}{|{\rm RA}(\gamma)|}
\end{equation}
for $f_1,f_2\in M(U_{{\rm npr}})$ where ${\rm RA}(\gamma):=
(H(\Q)\cap \gamma  U_{{\rm npr}} \gamma^{-1})Z(\A_f)/Z(\A_f)$.
Let $\varphi$ be a non-zero constant function on $H(\A_\Q)$.  
We denote by $M(U_{{\rm npr}})^0$ the orthogonal complement of $\C\varphi$ in $M(U_{{\rm npr}})$. 
Clearly, ${\rm dim}(M(U_{{\rm npr}})^0)=h-1=H_2(1,p)-1$. 
\begin{Def}\label{AMF} 
Each element of $M(U_{{\rm npr}})$ is said to be an algebraic modular form on $H(\A_\Q)=GUSp_4(\A_\Q)$ 
of weight zero with level $U_{{\rm npr}}$. 
\end{Def}

For each prime $\ell\neq p$ we define the (unramified) Hecke algebra 
$$\mathcal{H}_\ell=\C[H(\Z_\ell)\bs H(\Q_\ell)/H(\Z_\ell)]
\stackrel{(\ref{Ghitza})}{\simeq} 
\C[\GSp_4(\Z_\ell)\bs \GSp_4(\Q_\ell)/\GSp_4(\Z_\ell)]$$ at $\ell$ 
which is generated by the characteristic functions of form 
$H(\Z_\ell)gH(\Z_\ell)$ for $g\in H(\Q_\ell)$. 
Let $e_\ell$ be the characteristic function of $H(\Z_\ell)$ which is the identity element of 
$\mathcal{H}_\ell$. 
Let $\mathbb{T}^{(p)}=\otimes'_{\ell\neq p}\mathcal{H}_\ell$ be the 
restricted tensor product of $\{\mathcal{H}_\ell\}_{\ell\neq p}$ with respect to the 
identity elements $\{e_\ell\}_{\ell\neq p}$. 
We call $\mathbb{T}^{(p)}$ the Hecke ring outside $p$ and 
it is well-known that $\mathbb{T}^{(p)}$ acts on $M(U_{{\rm npr}})$ and also on $M(U_{{\rm npr}})^0$ (cf. Section 6 of \cite{Gross}).
 
\begin{Def}\label{HEAMF}\upshape{
Each element of $M(U_{{\rm npr}})$ is said to be a Hecke eigenform outside $p$ if 
it is a simultaneous eigenform for all elements in $\mathbb{T}^{(p)}$. }
\end{Def}
By using the Hermitian paring (\ref{pairing-alg}), 
we can show that there exists an orthonormal basis $HE(U_{{\rm npr}})$ of $M(U_{{\rm npr}})^0$ which 
consists of Hecke eigenforms outside $p$. 
For each non-zero $F$ in $HE(U_{{\rm npr}})$ and an element $T\in \mathbb{T}^{(p)}$, 
we denote by $\lambda_F(T)$ the eigenvalue of $F$ for $T$. 
Since $F$ has the trivial central character, $\lambda_T(F)$ is a real number. 
Let  $T(\ell)$ be the characteristic function of 
$H(\Z_\ell)t^H_\ell H(\Z_\ell)$ where 
$t^H_\ell=\diag(1,\ell)$ corresponds to $\diag(1,1,\ell,\ell)\in \GSp_4(\Q_\ell)$ under 
(\ref{Ghitza}). 
The element $T(\ell)\in \mathcal{H}_\ell$ is called the Hecke operator at $\ell$. 
Explicitly, if we write $H(\Z_\ell)t_\ell H(\Z_\ell)=\coprod_i g_i H(\Z_\ell)$, for each 
$F\in M(U_{{\rm npr}})$, its action is given by 
$$T(\ell)F=\sum_i R(g_i)F$$
where $R(g_i)F(h)=F(hg_i)$ for $h\in H(\A_\Q)$. Comparing to (\ref{Hecke-ope2}), 
$T(\ell)$ corresponds to $T(\ell)^{{\rm BT}}$ through  the natural identifications: 
\begin{equation}\label{BTU}
\G_{{\rm npr}}^{H}\bs H(\Q_\ell)/H(\Z_\ell) 
= H(\Q)\bs H(\A_f)/U_{{\rm npr}}=H(\Q)\bs H(\A_\Q)/(U_{{\rm npr}}\times H(\R)^+).
\end{equation}

As a link to the previously studied graphs, we obtain the following result, which
follows almost immediately from the definition of $M(U_{{\rm npr}})$ together with 
Corollary \ref{comp1} and Theorem \ref{Rel-HP}. We therefore omit the proof.
\begin{prop}\label{comp2}
The identifications of Corollary \ref{comp1} with Theorem \ref{Rel-HP} 
and (\ref{BTU}) yield the following isomorphisms 
$$
\ell^2({\mathcal G}^{{\rm MB}}(\ell, p))\simeq \ell^2({\rm BTQ}^1_g(\ell,p))
\stackrel{(\ref{BTU})}{\simeq} M(U_{{\rm npr}})
$$
as a finite Hilbert space such that the random walk operators 
$\frac{1}{N_2(\ell)}T(\ell)^{{\rm geo}}$ on $\ell^2({\mathcal G}^{{\rm MB}}(\ell, p))$ and 
$\frac{1}{N_2(\ell)}T(\ell)^{{\rm BT}}$ on $\ell^2({\rm BTQ}^1_g(\ell,p))$ correspond to
the normalized Hecke operator $\frac{1}{N_2(\ell)}T(\ell)$ on $M(U_{{\rm npr}})$ where 
$N_2(\ell)=(\ell^2+1)(\ell+1)$. 
\end{prop}

We denote by $S_3(K(p))$ the space of all Siegel paramodular cusp forms of weight three with level $p$ where 
$K(p)\subset \GSp_4(\Q)$ is the paramodular group of level $p$ (see \cite[Section 2]{Ibu-dim}). 
Using the identification $\mathcal{H}_\ell=\C[\GSp_4(\Z_\ell)\bs \GSp_4(\Q_\ell)/\GSp_4(\Z_\ell)]$ for each prime $\ell\neq p$, 
one can act $\mathcal{H}_\ell$ on $S_3(K(p))$ in the usual way.
We also denote by $S_3(K(p))^{{\rm CAP}}$ the space generated by Hecke eigen CAP forms. 
It is known that any Hecke eigen CAP form in  $S_3(K(p))$ is of Saito-Kurokawa type (see \cite{Schmidt}). Let $S_3(K(p))^{{\rm temp}}$ be the orthogonal complement of 
$S_3(K(p))^{{\rm CAP}}$ with respect to the Petersson inner product.  It follows from \cite[Theorem 3.2]{PS} with 
\cite{Wei}, for any Hecke eigen form $F$ in  $S_3(K(p))^{{\rm temp}}$, 
the corresponding cuspidal representation $\pi_F$ is tempered everywhere.   
Let $S_4(\G_0(p))^{-}$ be the space of elliptic cusp forms of weight 4 with 
respect to $\G_0(p)$ with the Atkin-Lehner eigenvalue $-1$. 
It also follows from \cite[Theorem 5.3 and its proof]{Schmidt} that 
there is a $\C$-linear isomorphism 
$${\rm SK}:S_4(\G_0(p))^{-}\lra S_3(K(p))^{{\rm CAP}}$$
sending Hecke eigenforms to Hecke eigenforms. 
Thus, we have an estimation
\begin{equation}\label{SKest}
{\rm dim}(S_3(K(p))^{{{\rm CAP}}})={\rm dim}(S_4(\G_0(p))^{-})\le {\rm dim}(S_4(\G_0(p)))\le 
3\Big(\frac{p+1}{12}-1\Big)+8=\frac{p+21}{4}
\end{equation}
by the dimension formula \cite[p.88]{DS}. 
Further, if $f=\ds\sum_{n>0}a_n(f)q^n\in S_4(\G_0(p))^{-}$ is a newform, then the eigenvalue of 
${\rm SK}(f)$ for 
$T(\ell)$ is given by 
\begin{equation}\label{evSK}
\lambda_{{\rm SK}(f)}(T(\ell)):=\ell^2+\ell+a_\ell(f).
\end{equation}
(see \cite{S1} or \cite{S2}). Notice that $|a_\ell(f)|\le 2\ell^{\frac{4-1}{2}}=2\ell\sqrt{\ell}$ by the Ramanujan bound. 

\begin{rmk}\label{dimformula}The space $S_3(K(p))^{{\rm CAP}}$ is generated by 
all forms which do not satisfy Ramanujan conjecture. Such forms 
are minor among the whole space $S_3(K(p))$ since the main term of 
${\rm dim}(S_3(K(p)))$ is $\ds\frac{p^2-1}{2880}$ whereas ${\rm dim}(S_3(K(p))^{{\rm CAP}})=O(p)$ 
by {\rm(}\ref{SKest}{\rm)}. 

By the dimension formula \cite[p.88]{DS}, we see $S_4(\G_0(2))=0$ and if $p>2$, 
$${\rm dim}(S_4(\G_0(p)))=\frac{1}{4}\{p-2+(-1)^{\frac{p-1}{2}}\}.$$
Further, by using \cite[Theorem 4]{Popa}, if $p>2$, we see that 
$${\rm tr}W_p|S_4(\G_0(p))=
\left\{\begin{array}{cl}
\frac{1}{2}h(\Q(\sqrt{-4p})) & (p\equiv 1,5\ {\rm mod}\ 8) \\
2h(\Q(\sqrt{-p})) & (p\equiv 3\ {\rm mod}\ 8) \\
h(\Q(\sqrt{-p})) & (p\equiv 7\ {\rm mod}\ 8)
\end{array}\right.
$$
where $W_p$ is the Atkin-Lehner involution at $p$ and $h(\Q(\sqrt{-d}))$ is the class number of the imaginary quadratic field of the 
fundamental discriminant $-d$. 
Then, ${\rm dim}(S_4(\G_0(p))^-)$ is computed by the formula 
$${\rm dim}(S_4(\G_0(p))^-)=\frac{1}{2}\{{\rm dim}(S_4(\G_0(p)))-{\rm tr}W_p|S_4(\G_0(p))\}.$$
\end{rmk}

The following claim is conjectured by Ibukiyama \cite[Conjecture 4.4]{Ibu-dim} and is 
proved by van Hoften \cite[Theorem 8.2.1]{vH}. 
\begin{thm}\label{IvH} There exists a $\C$-linear isomorphism $S_3(K(p))\simeq M(U_{{\rm npr}})^0$ which is compatible with the action of $\mathbb{T}^{(p)}$. 
\end{thm}

\begin{Def} A Hecke eigenform in $M(U_{{\rm npr}})^0$ is said to be a CAP form (resp. a non-CAP form) 
if the corresponding form is a CAP form (resp. the corresponding form belongs to 
$S_3(K(p))^{{\rm temp}}$) under $S_3(K(p))\simeq M(U_{{\rm npr}})^0$. 
\end{Def}

\begin{thm}\label{asymp-rel-ramanujan}
Put $d_{p}:={\rm dim}M(U_{{\rm npr}})^0=|HE(U_{{\rm npr}})|=H_2(1,p)-1$. 
For each prime $\ell\neq p$, it holds that
$$\limsup_{p\to \infty}\frac{1}{d_p}\sum_{F\in HE(U_{{\rm npr}})}|\lambda_F(T(\ell))|
\le 4 \ell^{\frac{3}{2}}.$$
\end{thm}
The upper bound is nothing but the Ramanujan bound for $T(\ell)$ for 
holomorphic Siegel Hecke eigen cusp forms on $GSp_4$ of weight 3 whose automorphic representations are tempered at $\ell$ (see Section 19 of \cite{vdGeer}). 
It also coincides with the spectral radius of the special $1$-complex $S_2$ (see 
\cite[Proposition 2.6]{Setyadi}).

By Theorem \ref{IvH} we can estimate the Hecke eigenvalues for each element of $M(U_{{\rm npr}})^0$.
For a Hecke eigenform $F$ in $M(U_{{\rm npr}})^0$, using the Satake parameters $\alpha_{\ell}$ and 
$\beta_{\ell}$ (see \cite[Section 18-19]{vdGeer}), we have 
$$
\lambda_{F}(T(\ell)) = \ell^{3/2}(\alpha_{\ell} + \alpha_{\ell}^{-1} + \beta_{\ell} 
+ \beta_{\ell}^{-1}).
$$
First, for each CAP Hecke eigenform $F$ in $M(U_{{\rm npr}})^0$,  we have 
$\vert \alpha_{\ell}\vert = 1$ and $\beta_{\ell} = \ell^{1/2}$ by (\ref{evSK}). Therefore, it satisfies 
\begin{equation}\label{CAP}
\ell^2+\ell -2\ell\sqrt{\ell}\le \lambda_{F}(T(\ell))\le \ell^2+\ell +2\ell\sqrt{\ell}.
\end{equation}
This also follows directly from (\ref{evSK}) with the Ramanujan bound. 

Next, for each non-CAP Hecke eigenform $F$ in $M(U_{{\rm npr}})^0$, we have 
$\vert \alpha_{\ell}\vert = \vert \beta_{\ell}\vert = 1$. Therefore, it satisfies  
\begin{equation}\label{non-CAP}
|\lambda_F(T(\ell))|\le 4\ell\sqrt{\ell}.
\end{equation}
As a byproduct we show the following results. 
\begin{prop}\label{evest} The graph ${\mathcal G}^{{\rm MB}}(\ell, p)$ is not bipartite. 
\end{prop}
\begin{proof}By (\ref{CAP}) and (\ref{non-CAP}), it is easy to see that 
the eigenvalues of $\frac{1}{(\ell^2+1)(\ell+1)}T(\ell)^{{\rm geo}}$ can not attain 
$-1$. The claim follows (cf. \cite[Exercise 1.1.1-(ii)]{Tao}). 
\end{proof}

Let $1=\mu_1>\mu_2\ge \cdots \ge \mu_h>-1$ be the eigenvalues of the 
random walk matrix (the normalized adjacency matrix) for $\mathcal{G}^{{\rm MB}}(\ell, p)$ with 
$h=H_2(1,p)$ and put $\lambda_i=1-\mu_i$.  
\begin{thm}\label{refinement}For each primes $\ell\neq p$ and $2\le i\le m$, it holds that 
$$
0.22287\ldots = 1 -\frac{6 + 4\sqrt{2}}{15} \le 1-\frac{\ell^2+\ell+2\ell\sqrt{\ell}}{N_{2}(\ell)}\le \lambda_i \le 
1+\frac{4\ell\sqrt{\ell}}{N_{2}(\ell)} \le 1 +\frac{8\sqrt{2}}{15} = 1.75424\ldots
$$
where $N_2(\ell)=(\ell^2+1)(\ell+1)$. 
\end{thm}

\begin{rmk}Let $\lambda_\star=\min\{\lambda_2,2-\lambda_h\}$ and 
$\mu_\star:=\max\{|\mu_2|,|\mu_h|\}$ so that $\lambda_\ast=1-\mu_\ast$. 
Then, the both quantities are important to estimate the total variation mixing time 
$t_{{\rm mix}}(\ve)$ for 
the random walk on $\mathcal{G}^{{\rm MB}}(\ell, p)$ 
{\rm(}see \cite[Section 4.5, (4.30)]{LP} for $t_{{\rm mix}}(\ve)$ and note that our graph is reversible since 
the Markov operator is self-adjoint{\rm)}. 
In fact, by Theorem \ref{refinement}, we have 
$$
\lambda_\star\ge 1-\frac{\ell^2+\ell+2\ell\sqrt{\ell}}{N_{2}(\ell)}\ge 0.22287\ldots.
$$
Thus, a well-known estimation of the mixing time (see \cite[Theorem 12.4]{LP}) shows 
$$
t_{{\rm mix}}(\ve)\le \frac{1}{\lambda_\star}\log\Big(\frac{1}{2\ve \sqrt{\pi_{\min}}}\Big)
=4.486792\ldots\times\log\Big(\frac{1}{2\ve \sqrt{\pi_{\min}}}\Big),\ 0< {}^\forall\ve\le 1
$$
where $\pi_{\min}:=\min\{\pi(v)\ |\ v\in V(\mathcal{G}^{{\rm MB}}(\ell, p)) \}$ 
and $\pi$ is the stationary distribution defined by 
$$\pi(v):=
\frac{|{\rm RA}(v)|^{-1}}{\ds\sum_{ w\in V(\mathcal{G}^{{\rm MB}}(\ell, p))}|{\rm RA}(w)|^{-1}},\ 
v\in V(\mathcal{G}^{{\rm MB}}(\ell, p)).$$ 
It follows from a simple estimation that 
$$
\pi_{\min}\ge \ds\frac{1}{H_2(1,p)\max\{|{\rm RA}(v)|\ |\ v\in V(\mathcal{G}^{{\rm MB}}(\ell, p))\}}\ge  \frac{1}{120 H_2(1,p)}.
$$ 
Here we used \cite[Theorem 7.1 and p. 343 Remark 1]{Ibu-aut} to check 
$|{\rm RA}(v)|\le |{\rm PGL}_2(\bbF_5)|=120$ 
for any $v\in V(\mathcal{G}^{{\rm MB}}(\ell, p))$.

Summing up, we have 
$$t_{{\rm mix}}(\ve)\le 4.486792\ldots\times\log\Big(\frac{\sqrt{120 H_2(1,p)}}{2\ve }\Big).$$
\end{rmk}

We end this section with the following tables showing the numbers of Hecke eigen CAP forms 
listed in the third row  and Hecke eigen non-CAP forms  listed in the fourth row together with $d_p:={\rm dim}
(M(U_{{\rm npr}})^0)=H_2(1,p)-1$. The former is simply the dimension of $S_4(\Gamma_0(p))^{-}$, which can be computed by the formula in Remark \ref{dimformula} 
(or, as a check, it can be also read off from the database \cite{LMFDB}):
\begin{table}[htbp]
{\tiny
\centering 
\renewcommand{\arraystretch}{1.5}
\begin{tabular}{|c|c|c|c|c|c|c|c|c|c|c|c|c|c|c|c|c|c|c|c|c|c|c|c|} \hline
$p$ & $\le 11$ & $13$ & $17$  & $19$ & $23$ & $29$ & $31$ &  $37$ & $41$ & $43$ 
& $47$ & $53$  & $59$  & $61$ & 67  & 71 & 73 & 79 & 83 & 89 & 97 &$\cdots$& 2027  \\ 
\hline
$d_p$ & $0$ & $1$ & $1$  & $1$ & $1$ & $2$ & $2$ &  $4$ & $3$ & $4$
& $3$ & $5$ & $4$ & $7$  & 7 & 5& 9& 8&7&9&13 &$\cdots$& 1546 \\
 \hline
CAP & $0$ & $1$ & $1$  & $1$ & $1$ & $2$ &  $2$  &  $4$ & $3$ & $4$  
& $3$ &  $5$& 4 & 6 &  7 &5& 8&7 &7&8&11 &$\cdots$& $242$\\
\hline
non-CAP & $0$ & $0$ & $0$  & $0$ & $0$ & $0$ & $0$  &  $0$ & $0$ & $0$
& 0 &0& 0  & 1 & 0 &0& 1&1 &0&1&2 & $\cdots$&1304 \\ \hline
\end{tabular}
}
\medskip
\caption{${\rm dim}(M(U_{{\rm npr}})^0)$, the number of Hecke eigen CAP (or non-CAP) forms.}
\label{table2}
\end{table}

\section{Matrix coefficients and a Sarnak-Xue type theorem}\label{MCSX}
Let $\ell$ be a prime different from $p$. 
We denote $\G_{{\rm npr}}^H(p):=\G_{{\rm npr}}^H=H(\Q)\cap K(L_{{\rm npr}})^{(\ell)}$ and 
$\G_{{\rm npr}}^{\GSp_4}(p):=\G_{{\rm npr}}^{\GSp_4}$ to emphasize the dependence of $p$.
We fix the Haar measure $\mu$ on  $H(\Q_\ell)/Z_H(\Q_\ell)\simeq \PGSp_4(\Q_\ell)$ 
such that the volume of the maximal open compact subgroup $(H/Z_H)(\Z_\ell)\simeq \PGSp_4(\Z_\ell)$ is one. It naturally defines a measure on  
the quotient space $\G_{{\rm npr}}^H(p)\bs H(\Q_\ell)/Z_H(\Q_\ell)\simeq 
\G_{{\rm npr}}^{\GSp_4}(p)\bs \PGSp_4(\Q_\ell)$ since $\mu$ is left invariant. 
We denote by $L^2(\G_{{\rm npr}}^H(p)\bs H(\Q_\ell)/Z_H(\Q_\ell))$ the $L^2$-space of 
$\G_{{\rm npr}}^H(p)\bs H(\Q_\ell)/Z_H(\Q_\ell)$ with respect to $\mu$. 
The $L^2$-space of 
$\G_{{\rm npr}}^{\GSp_4}(p)\bs \PGSp_4(\Q_\ell)$ is also similarly defined with respect to 
$\mu$. 
The groups $H(\Q_\ell)$ and $H(\Q_\ell)/Z_H(\Q_\ell)$ act there as the right translation. 
By (\ref{JZ}), $\G_{{\rm npr}}(p)\bs H(\Q_\ell)/Z_H(\Q_\ell)$ is a finite disjoint union of 
$(H/Z_H)(\Z_\ell)\simeq\PGSp_4(\Z_\ell)$ and thus, it is compact. By the Stone–Weierstrass Theorem 
(cf. \cite[p.42, Theorem 2.40]{EW}), the subspace consisting of all locally constant functions 
is dense in $L^2(\G_{{\rm npr}}(p)\bs H(\Q_\ell)/Z_H(\Q_\ell))$. Under (\ref{Ghitza}), we also have 
$$
L^2(\G_{{\rm npr}}(p)^H\bs H(\Q_\ell)/Z_H(\Q_\ell))\simeq 
L^2\!\left(\G_{{\rm npr}}^{\PGSp_4}(p)\backslash \PGSp_4(\Q_\ell)\right).$$
Thus, we have 
$$
M(U_{{\rm npr}})=L^2(\G_{{\rm npr}}^H(p)\bs H(\Q_\ell)/Z_H(\Q_\ell))^{H(\Z_\ell)}
\simeq L^2(\G_{{\rm npr}}^{\PGSp_4}(p)\bs \PGSp_4(\Q_\ell))^{\PGSp_4(\Z_\ell)}.
$$
The space $L^2(\G_{{\rm npr}}^H(p)\bs H(\Q_\ell)/Z_H(\Q_\ell))$ looks like local objects but 
it relates to global objects as $M(U_{{\rm npr}})$ via (\ref{JZ}). 
Thus, local and global representations intervene there. 

Let $\tau$ be an irreducible admissible unitary representation of $H(\Q_\ell)/Z_H(\Q_\ell)\simeq\PGSp_4(\Q_\ell)$ 
(see \cite[p.18, Section 2.1]{Ca-book} for terminology).  
Let $\tau^\vee$ be the contragredient representation of $\tau$.  
Let $V$ (resp. $V^\vee$) be the representation space of $\tau$ 
(resp. $\tau^\vee$) and $\langle \ast,\ast\rangle$ be the natural pairing on 
$V\times V^\vee$.  
For each $v\in V$ and $w\in V^\vee$, we define the matrix coefficient of $\tau$
associated with $(v,w)$ by
\[
c_{v,w}\colon H(\Q_\ell)/Z_H(\Q_\ell)\longrightarrow \C,\ g\longmapsto \langle gv,w\rangle.
\]
For each $q\in \R_{>0}$, we say that \emph{all matrix coefficients of $\tau$ belong to
$L^q$} if
\[
c_{v,w}\in L^q\bigl(H(\Q_\ell)/Z_H(\Q_\ell)\bigr)
\quad\text{for all } v\in V \text{ and } w\in V^\vee.
\]
The reader should not confuse this notion with that even if $\tau$ appear in 
$L^2\bigl(\G_{\mathrm{npr}}^H(p)\bs H(\Q_\ell)/Z_H(\Q_\ell)\bigr)$, 
it is unrelated to the condition that all matrix coefficients belong to $L^2$. 
We say $\tau$ is tempered if  all matrix coefficients belong to $L^{2+\ve}$ for 
any $\ve>0$ and non-tempered otherwise.  
Non-tempered representations are related to the phenomena violating 
Ramanujan conjecture (cf. \cite{Sarnak}).

It is important to understand what kinds of $\tau$ appears in 
$L^2(\G_{{\rm npr}}^H(p)\bs H(\Q_\ell)/Z_H(\Q_\ell))$. When $\tau$ is non-supercuspidal 
(this is the case in our setting), such a representation 
is completely classified by \cite{RS}.  
We denote by $m(\tau,\G_{{\rm npr}}^H(p))$ the multiplicity of $\tau$ in 
$L^2(\G_{{\rm npr}}^H(p)\bs H(\Q_\ell)/Z_H(\Q_\ell))$. 
We also define $q(\tau)$ to be the infimum of $q\in \R_{>0}$ such that 
any matrix coefficients of $\tau$ belong to $L^q$.

We denote by $\Pi_{{\rm sph}}$ the subset of the unitary dual of 
$H(\Q_\ell)/Z_H(\Q_\ell)\simeq \PGSp_4(\Q_\ell)$ which consists of all unramified 
irreducible admissible uniatry representations. Here $\tau$ is said to be unramified 
if it has a non-zero $(H/Z_H)(\Z_\ell)\simeq \PGSp_4(\Z_\ell)$-fixed vector.   
For each positive real number $q$, we define
\begin{equation}\label{multiLq}
M(\Pi_{{\rm sph}},\G_{{\rm npr}}^H(p),q):=\sum_{\tau\in \Pi_{{\rm sph}}\atop q(\pi)\ge q}
m(\tau,\G_{{\rm npr}}(p)).
\end{equation}
Each $\tau \in \Pi_{{\rm sph}}$ has a non-zero $(H/Z_H)(\Z_\ell)\simeq 
\PGSp_4(\Z_\ell)$-fixed vector and 
thus, the multiplicity $m(\tau,\G_{{\rm npr}}(p))$ coincides 
with the multiplicity of $\tau$ as a $\mathcal{H}_\ell$-module in 
$$L^2(\G_{{\rm npr}}^H(p)\bs H(\Q_\ell)/Z_H(\Q_\ell))^{(H/Z_H)(\Z_\ell)}=M(U_{{\rm npr}}).$$ 
Then we prove the following statement:
\begin{thm}\label{SX}
It holds that 
\begin{enumerate}
\item If $0<q\le 2$, then $M(\Pi_{{\rm sph}},\G_{{\rm npr}}^H(p),q)=
{\rm dim}(M(U_{{\rm npr}})^0)=H_2(1,p)-1=\ds\frac{p^2-1}{2880}+O(p)$;
\item If $2<q\le 3$, then $M(\Pi_{{\rm sph}},\G_{{\rm npr}}^H(p),q)=
{\rm dim}(S_3(K(p))^{{\rm CAP}})={\rm dim}(S_4(\G_0(p))^-)=O(p)$;
\item If $q>3$, then $M(\Pi_{{\rm sph}},\G_{{\rm npr}}^H(p),q)=0$.
\end{enumerate}
Further, for any $q\in \R_{}>0$, a Sarnak-Xue type hypothesis 
$$
\lim_{p\to \infty}
\frac{M(\Pi_{{\rm sph}},\G_{{\rm npr}}(p),q)}{
{\rm dim}(M(U_{{\rm npr}}))^{\frac{2}{q}}}
=0$$
holds. 
\end{thm}
\begin{proof}
Let $L^2_0(\G_{{\rm npr}}^H(p)\bs H(\Q_\ell)/Z_H(\Q_\ell))\simeq 
L^2_0(\G_{{\rm npr}}^{\GSp_4}(p)\bs \PGSp_4(\Q_\ell))$ be the 
orthogonal complement of the space of all constant functions. 
By Theorem \ref{IvH}, we see that 
$$L^2_0(\G_{{\rm npr}}^{\GSp_4}(p)\bs \PGSp_4(\Q_\ell))^{\PGSp_4(\Z_\ell)}\simeq L^2_0(\G_{{\rm npr}}^H(p)\bs H(\Q_\ell)/Z_H(\Q_\ell))^{H(\Z_\ell)}=  
M(U_{{\rm npr}})^0 \simeq 
S_3(K(p)).$$
Let $F$ be a non-zero Hecke eigen cusp form in $S_3(K(p))$ and $\pi=\pi_F$ be 
the irreducible unitary cuspidal automorphic representation of $\PGSp_4(\A_\Q)$ associated to $F$ 
(see \cite[p.196, Section 4.6]{BJ} for cuspidal automorphic representations). 
Note that $\pi_F$ uniquely determines $F$ up to scaling by the strongly multiplicity one 
(cf. \cite{Schmidt18}, \cite{WWYY}). By Flath's theorem \cite{Flath}, we can decompose 
$\pi=\otimes'_v \pi_v$ as a restricted tensor product where $\pi_v$ is an 
irreducible unitary admissible representation of $\PGSp_4(\Q_v)$ for 
each place $v$ of $\Q$.  Since $F\in S_3(K(p))$ and $p\neq \ell$,  
the $\ell$-th component $\tau=\pi_\ell$ is unramified. 
We shall compute the matrix coefficient of $\tau$ by using both of 
local and global arguments. 
%Henceforth we sometimes regard $\tau$ with an irreducible  unitary admissible representation of $H(\Q_\ell)$ with the trivial central character  through (\ref{Ghitza}).  

First of all, by \cite[Theorem 3.2]{PS} with \cite{Wei} (note that the case (C) therein can not happen), 
$\tau$ is either an unramified ``tempered'' principal series (of type I) or is 
an unramified non-tempered of type IIb in the sense of the classification in \cite{RS}. 
We have so far discussed the global argument; we now turn to a local argument.
For $\tau$ in the former type I, we have $q(\tau)=2$ since it is tempered. 
As for the second type IIb, we can apply the criterion 
\cite[Corollary 4.4.5]{Ca-book} to compute $q(\tau)$. Note that we use the notation 
in \cite[p.18]{RS} such that $GSp_4$ is defined 
by $J={\rm anti\text{-}diag}(1,1,-1,-1)$. Our computation goes as follows. 
Let $B=TU$ be the standard Borel subgroup of $GSp_4$ with the diagonal torus 
$T=\{\diag(a,b,c b^{-1},c a^{-1})\ |\ a,b,\nu\in GL_1\}\simeq (GL_1)^3$ and 
the unipotent radical $U$.  Let $\delta_B:B(\Q_{\ell})\lra \C^\times$ 
be the modulus character of $B(\Q_{\ell})$ defined by 
$$
\delta_B(\diag(a,b,c b^{-1},c a^{-1})u)=|a^4b^2c^{-3}|_\ell
$$ 
for each 
$\diag(a,b,c b^{-1},c a^{-1})u\in B(\Q_{\ell})=T(\Q_{\ell})U(\Q_{\ell})$. 

Let $\tau_U$ be the Jacquet-module of $\tau$ along $U$ 
(cf. \cite[p.34]{Ca-book}). It is an admissible representation of $T(\Q_\ell)$ but may be reducible. 
For each $t=\diag(a,b,c b^{-1},c a^{-1})\in T(\Q_\ell)$, we define $\alpha,\beta:
T(\Q_\ell)\lra \Q^\times_\ell$ by 
$$
\alpha(t)=ab^{-1},\ \beta(t)=b^2c^{-1}.
$$
Note that $\alpha$ and $\beta$ are simple roots of $GSp_4$.
Put  
$$
T(\Q_\ell)^-:=\{t\in T(\Q_\ell)\ |\ |\alpha(t)|_\ell\le 1,\ |\beta(t)|_\ell \le 1\}.
$$
It is easy to see that $T(\Q_\ell)^-/T(\Z_\ell)Z_{GSp_4}(\Q_\ell)$ is generated by 
$$t_1:=\diag(\ell,\ell,1,1),\ t_2=\diag(\ell^2,\ell,\ell,1).$$
The Casselman's criterion (\cite[Corollary 4.4.5]{Ca-book}) for $\tau$ to belong 
$L^q$ ($q\in \R_{>0}$) is described  in terms of  the action of $T(\Q_\ell)^-/T(\Z_\ell)Z_{GSp_4}(\Q_\ell)$ on 
$(\tau_U)^{T(\Z_\ell)}=((\tau_U)^{{\rm ss}})^{T(\Z_\ell)}$. Here the superscript ``ss'' stands for the semisimplification. 
Then, the action of $T(\Q_\ell)^-/T(\Z_\ell)Z_{GSp_4}(\Q_\ell)$ can be computed by using the decomposition 
of $(\tau_U)^{{\rm ss}}$ given explicitly in \cite[p.541, Case IIb]{MY} (see also 
\cite[p.540, the line -2]{MY} for $\chi_1$ and $\chi_2$). 
In fact, as a representation of $T(\Q_p)$, we have 
$$
(\tau_U)^{{\rm ss}}=\chi_2\otimes \chi_1 \otimes \sigma +
\chi^{-1}_1\otimes \chi^{-1}_2 \otimes \chi_1\chi_2\sigma +
\chi_2\otimes \chi^{-1}_1 \otimes \chi_1\sigma +
\chi^{-1}_1\otimes \chi_2 \otimes \chi_1\sigma,
$$
where $\chi_1:=|\ast|^{\frac{1}{2}}_\ell \chi,\ \chi_2:=|\ast|^{-\frac{1}{2}}_\ell \chi$ 
for some unramified unitary characters $\chi,\sigma$ of $\Q^\times_\ell$ satisfying 
$\chi_1\chi_2\sigma^2=\mathbf{1}$. 
The character ``$\chi$'' in \cite[Corollary 4.4.5]{Ca-book} for $\tau_U$ runs over 
the above four characters in the decomposition of $(\tau_U)^{{\rm ss}}$. 
The Casselman's criterion requires 
to check if $$|\chi(a)| \delta^{\frac{1}{2}-\frac{1}{q}}_B(a)<1$$ 
for any non-trivial element $a\in T(\Q_\ell)^-/T(\Z_\ell)Z_{GSp_4}(\Q_\ell)$.  
Here $|\chi(a)|$ stands for the absolute value of the complex number $\chi(a)$.  
We remark that Casselman considered $|\chi(a)|\delta^{-\frac{1}{q}}_B(a)<1$ therein, but 
this is just because he considered non-normalized inductions and the difference   
is just $\delta^{\frac{1}{2}}_B$.
We also note that $|\chi (a)| =|\sigma (a)| = |\chi_1\chi_2(a)|= 1$, since $\chi$ and $\sigma$
are unitary.  
 Since $\delta_B(t_1)=\ell^{-3},\ \delta_B(t_2)=\ell^{-4}$, we have the following inequalities among eight tests 
 (four characters in $\chi$ with two generators $t_1,t_2$) as  
$$
 -\frac{3}{2}+\frac{3}{q}<0,\ -\frac{1}{2}-\frac{3}{2}+\frac{3}{q}<0,\ \frac{1}{2}-\frac{3}{2}+\frac{3}{q}<0
$$
 and 
$$
       \frac{1}{2}-\frac{4}{2}+\frac{4}{q}<0.
$$
Therefore, we must have $q>3$ and thus, $q(\tau)=3$.  

The Sarnak-Xue type hypothesis in the statement follows by observing  
the main term of $H_2(1,p)$ (that is $\ds\frac{p^2-1}{2880}$) and  (\ref{SKest}). 
\end{proof}

\end{document}